\documentclass[twoside]{article}
\usepackage[cp1251]{inputenc}
\usepackage[T2A]{fontenc}
\usepackage{array}
\usepackage{amssymb}
\usepackage{amsmath}
\usepackage{amsthm}
\usepackage{latexsym}
\usepackage{indentfirst}
\usepackage{bm}
\usepackage{enumerate}
\usepackage[dvips]{graphicx}
\usepackage{epsf}
\usepackage{wrapfig}
\usepackage{euscript}
\usepackage{indentfirst}
\usepackage[english,russian]{babel}
\usepackage{textcomp}
\newtheorem{theorem}{Theorem}
\newtheorem{lemma}{Lemma}
\newtheorem{definition}{Definition}

\newtheorem{corollary}{Corollary}
\begin{document}
{\selectlanguage{english}
\binoppenalty = 10000 %
\relpenalty   = 10000 %

\pagestyle{headings} \makeatletter
\renewcommand{\@evenhead}{\raisebox{0pt}[\headheight][0pt]{\vbox{\hbox to\textwidth{\thepage\hfill \strut {\small Grigory. K. Olkhovikov}}\hrule}}}
\renewcommand{\@oddhead}{\raisebox{0pt}[\headheight][0pt]{\vbox{\hbox to\textwidth{{Characterization of predicate intuitionistic formulas}\hfill \strut\thepage}\hrule}}}
\makeatother

\title{Model theoretic-characterization of predicate intuitionistic formulas}
\author{Grigory K. Olkhovikov\\ Department of Ontology and Cognition Theory\\ Ural Federal University\\
Fulbright Visiting Scholar at the Philosophy Dept,\\Stanford University \\Bldg 90, Stanford, CA, USA}
\date{}
\maketitle
\begin{quote}
{\bf Abstract.} Notions of asimulation and $k$-asimulation
introduced in \cite{Ol} are extended onto the level of predicate
logic. We then prove that a first-order formula is equivalent to a
standard translation of an intuitionistic predicate formula iff it
is invariant with respect to $k$-asimulations for some $k$, and
then that a first-order formula is equivalent to a standard
translation of an intuitionistic predicate formula iff it is
invariant with respect to asimulations. Finally, it is proved that
a first-order formula is equivalent to a standard translation of
an intuitionistic predicate formula over a class of intuitionistic
models (intuitionistic models with constant domain) iff it is
invariant with respect to asimulations between intuitionistic
models (intuitionistic models with constant domain).
\end{quote}

Van Benthem's well-known  modal characterization theorem shows
that expressive power of modal propositional logic as a fragment
of first-order logic can be described via the notion of
bisimulation invariance. Moreover, it is known that modal
predicate logic, initially considered as an extension of
first-order logic, can also be viewed as its fragment, although
somewhat bigger than the fragment induced by propositional modal
logic. Expressive power of modal predicate logic, from this
vantage point, is described by the notion of world-object
bisimulation which appears to be a rather direct combination of
bisimulation and partial isomorphism (see, e.\,g. \cite[p. 124,
Theorem 21]{vB}).

Although  intuitionistic logic has been treated as a fragment of
modal logic for quite a long while, results analogous to
propositional and predicate version of Van Benthem's modal
characterization theorem were not obtained for it until recently.
In \cite{Ol} we filled this gap for intuitionistic propositional
logic. In this paper we introduced the notion of asimulation and
its parametrized version, $k$-asimulation, and showed that they
can be used to characterize expressive power of intuitionistic
propositional logic in much the same way bisimulation and
$k$-bisimulation are used to characterize modal propositional
logic. In this paper we do the same job for intuitionistic
predicate logic without identity.

The layout of the paper is as follows. Starting from some
notational conventions and preliminary remarks in section
\ref{S:Prel}, we then define a predicate version of
$k$-asimulation and move on to the proof of a `parametrized'
version of model-theoretic characterization of intuitionistic
predicate logic in section \ref{S:Param}. Then, in section
\ref{S:Main}, we introduce the predicate version of asimulation
and prove the full unparametrized counterpart to Theorem 21 of
\cite{vB}. In section \ref{S:Rest} we discuss possibilities of
restriction of the latter result to special subclasses of
first-order models and the final sections contains some
conclusions, and mentions possible directions of further research.

\section{Preliminaries}\label{S:Prel}

We take $\mathbb{N}$ to be the set of natural numbers
\emph{without} $0$. A formula is a formula of classical predicate
logic with identity whose predicate letters are in a vocabulary
$\Sigma = \{\,R^2,E^2\,\} \cup \{\,P^n_m \mid n,m
\in\mathbb{N}\,\}$, where the upper subscript denotes the arity of
the letter, so $0$-ary predicate letters or propositional letters
are not allowed. We refer to formulas with Greek letters distinct
from $\alpha$ and $\beta$, and to sets of formulas with upper-case
Greek letters distinct from $\Sigma$ and $\Theta$. We refer to
variables with letters $w, x, y, z$, sometimes using primes or
subscripts. If $\varphi$ is a formula, then we associate with it
the following finite vocabulary $\Sigma_\varphi \subseteq \Sigma$
such that $\Sigma_\varphi = \{\,R^2, E^2\,\} \cup \{\,P^j_i \mid
P^j_i \text{ occurs in }\varphi\,\}$. More generally, we refer
with $\Theta$ to an arbitrary subset of $\Sigma$ such that $R^2,
E^2 \in \Theta$. If $\psi$ is a formula and every predicate letter
occurring in $\psi$ is in $\Theta$, then we call $\psi$ a
$\Theta$-formula.

We refer to sequence $x_1,\dots, x_n$ of any objects as
$\bar{x}_n$. We denote ordered pair of ordered $n$-tuple
$(\bar{x}_n)$ and ordered $m$-tuple $(\bar{y}_m)$ by
$(\bar{x}_n;\bar{y}_m)$. We identify ordered $1$-tuple with its
only member. We denote the ordered $0$-tuple by $\Lambda$. If all
free variables of a formula $\varphi$ (set of formulas $\Gamma$)
are among $\bar{x}_n$, we write $\varphi(\bar{x}_n)$
($\Gamma(\bar{x}_n)$).

For a binary relation $S$ and any objects $s, t$ we abbreviate the
fact that $sSt \wedge tSs$ by $s\hat{S}t$.

We will denote models of classical predicate logic by letters $M$,
$N$ or $\alpha, \beta$. We refer to the domain of a model $M$ by
$D(M)$. For $n \geq 0$ by an $n$-ary evaluation $\Theta$-point we
mean a sequence $(M, a, \bar{b}_n)$ such that $M$ is a
$\Theta$-model and $(a,\bar{b}_n)$ is a sequence of elements of
$D(M)$. If $(M, a, \bar{b}_n)$ is an $n$-ary evaluation point then
we say that $\varphi(x, \bar{w}_n)$ is true at $(M, a, \bar{b}_n)$
and write $M, a, \bar{b}_n \models \varphi(x, \bar{w}_n)$ iff for
any variable assignment $f$ in $M$ such that $f(x) = a$, $f(w_i) =
b_i$ for any $1 \leq i \leq n$ we have $M, f \models \varphi(x,
\bar{w}_n)$. It follows from this convention that truth of a
formula $\varphi(x, \bar{w}_n)$ at an $n$-ary evaluation point is
to some extent independent of a choice of its free variables.

An intuitionistic formula is a formula of intuitionistic predicate
logic without identity. Propositional (i.\,e. $0$-ary predicate)
letters are allowed. We refer to intuitionistic formulas with
letters $i, j, k$, possibly with primes or subscripts.  Their
variables are represented in the same way as in formulas. We
assume a standard Kripke semantics for intuitionistic predicate
logic where in a given world a predicate letter might be true only
for some tuples of objects present in this world.

If $x$ is an individual variable in a first-order language, then
by a standard $x$-translation of intuitionistic formulas into
formulas we mean the following map $ST$ defined by induction on
the complexity of the corresponding intuitionistic formula. First
we assume some map of intuitionistic predicate letters into
classical ones which correlates with each $n$-ary intuitionistic
predicate letter $P$ an $(n + 1)$-ary classical predicate letter
$P'$ distinct from $R^2, E^2$. We assume that this correlation is
surjective, that is, that every predicate letter in $\Sigma$
distinct from $R^2, E^2$ is standard translation of an
intuitionistic predicate letter. Then our induction goes as
follows:
\begin{align*}
&ST(P(\bar{w}_n), x) = P'(x, \bar{w}_n);\\
&ST(\bot, x) = (x \neq x);\\
&ST(i(\bar{w}_n) \wedge j(\bar{w}_n), x) = ST(i(\bar{w}_n), x)
\wedge
ST(j(\bar{w}_n), x);\\
&ST(i(\bar{w}_n) \vee j(\bar{w}_n), x) = ST(i(\bar{w}_n), x) \vee
ST(j(\bar{w}_n), x);\\
&ST(i(\bar{w}_n) \to j(\bar{w}_n), x) = \forall y(R(x, y) \to
(ST(i(\bar{w}_n), y) \to ST(j(\bar{w}_n), y)));\\
&ST(\exists w'i(\bar{w}_n, w'), x) = \exists w'(E(x,w') \wedge
ST(i(\bar{w}_n, w'), x));\\
&ST(\forall w'i(\bar{w}_n, w'), x) = \forall yw'((R(x, y) \wedge
E(y, w')) \to ST(i(\bar{w}_n, w'), y)).
\end{align*}

Standard conditions are imposed on the variables $x, y, \bar{w}_n,
w'$.

By degree of a formula we mean the greatest number of nested
quantifiers occurring in it. A degree of a formula $\varphi$ is
denoted by $r(\varphi)$. Its formal definition by induction on the
complexity of $\varphi$ goes as follows:
\begin{align*}
&r(\varphi) = 0 &&\text{for atomic $\varphi$}\\
&r(\neg\varphi) = r(\varphi)\\
&r(\varphi \circ \psi) = max(r(\varphi), r(\psi)) &&\text{for $\circ \in \{\,\wedge, \vee, \to\,\}$}\\
&r(Qx\varphi) = r(\varphi) + 1 &&\text{for $Q \in \{\,\forall,
\exists\,\}$}
\end{align*}

If $k \in \mathbb{N}$ and $\varphi(x, \bar{w}_n)$ is a
$\Theta$-formula such that $r(\varphi) \leq k$, then $\varphi$ is
a $(\Theta, (x, \bar{w}_n), k)$-formula.

\section{Characterization of intuitionistic predicate formulas
via $k$-asimulations}\label{S:Param}

We begin with extending our previous notion of $k$-asimulation to
cover the general case of predicate logic.
\begin{definition}\label{D:k-asim}
Let $(M, a, \bar{b}_n)$, $(N, c, \bar{d}_n)$ be two $n$-ary
evaluation $\Theta$-points. A binary relation
\[
A \subseteq \bigcup_{m \geq 1, l \geq 0}(((D(M)^m \times D(M)^l)
\times (D(N)^m \times D(N)^l)) \cup ((D(N)^m \times D(N)^l) \times
(D(M)^m \times D(M)^l))),
\]
is called $\langle (M, a, \bar{b}_n), (N, c,
\bar{d}_n)\rangle_k$-asimulation iff $(a; \bar{b}_n)A(c;
\bar{d}_n)$ and for any $\alpha, \beta \in \{\,M, N\,\}$, any
$(\bar{a'}_m, a';\bar{b'}_l) \in D(\alpha)^{m+1} \times
D(\alpha)^l$, $(\bar{c'}_m, c';\bar{d'}_l) \in  D(\beta)^{m+1}
\times D(\beta)^l$, whenever we have $(\bar{a'}_m,
a';\bar{b'}_l)A(\bar{c'}_m, c';\bar{d'}_l)$, the following
conditions hold:

\begin{align}
&\forall P \in \Theta\setminus\{\,R^2,E^2\,\}(\alpha, a',
\bar{b'}_l\models P(x, \bar{w}_l)
\Rightarrow \beta, c', \bar{d'}_l\models P(x, \bar{w}_l))\label{E:c1}\\
&(m + l < n + k \wedge c'' \in D(\beta) \wedge c'R^\beta c'') \Rightarrow\notag\\
&\Rightarrow \exists a'' \in D(\alpha)(a'R^\alpha a'' \wedge
(\bar{c'}_m, c', c'';\bar{d'}_l)\hat{A}(\bar{a'}_m, a',
a'';\bar{b'}_l));\label{E:c2}\\
&(m + l < n + k \wedge b'' \in D(\alpha) \wedge E^\alpha(a', b''))  \Rightarrow\notag\\
&\Rightarrow \exists d'' \in D(\beta)(E^\beta(c',d'') \wedge
(\bar{a'}_m, a';\bar{b'}_l, b'')A(\bar{c'}_m,c';\bar{d'}_l, d''));\label{E:c3}\\
&(m + l + 1 < n + k \wedge  c'', d'' \in D(\beta) \wedge c'R^\beta c''\wedge E^\beta(c'', d'')) \Rightarrow\notag\\
&\Rightarrow \exists a'',b'' \in D(\alpha)(a'R^\alpha a''\wedge
E^\alpha(a'',b'') \wedge (\bar{a'}_m, a', a'';\bar{b'}_l,
b'')A(\bar{c'}_m,c', c'';\bar{d'}_l, d'')).\label{E:c4}
\end{align}
\end{definition}

\begin{lemma}\label{L:asim}
Let $\varphi(x, \bar{w}_n) = ST(i(\bar{w}_n), x)$ for some
intuitionistic formula $i(\bar{w}_n)$, and let $r(\varphi) = k$.
Let $\Sigma_\varphi \subseteq \Theta$, let $(M, t,\bar{u}_s)$,
$(N, t', \bar{u'}_s)$ be two $s$-ary evaluation $\Theta$-points,
and let $A$ be an $\langle(M, t,\bar{u}_s), (N, t',
\bar{u'}_s)\rangle_p$-asimulation. Then
\begin{align*}
&\forall\alpha, \beta \in \{\,M, N\,\}\forall(\bar{a}_m,
a;\bar{b}_n) \in (D(\alpha)^{m+1} \times
D(\alpha)^n)\forall(\bar{c}_m, c;\bar{d}_n)\in (D(\beta)^{m+1}
\times
D(\beta)^n)\\
&(((\bar{a}_m, a;\bar{b}_n)A(\bar{c}_m, c;\bar{d}_n) \wedge m + n
+ k \leq p + s \wedge \alpha, a,  \bar{b}_n\models \varphi(x,
\bar{w}_n) \Rightarrow \beta, c,\bar{d}_n \models \varphi(x,
\bar{w}_n)).
\end{align*}
\end{lemma}
\begin{proof} We proceed by induction on the complexity of $i$.
In what follows we will abbreviate the induction hypothesis by IH.

\emph{Basis.}  Let $i(\bar{w}_n) = P(\bar{w}_n)$. Then
$\varphi(x,\bar{w}_n) = P'(x,\bar{w}_n)$ and we reason as follows:
\begin{align}
&(\bar{a}_m, a;\bar{b}_n)A(\bar{c}_m, c;\bar{d}_n)\label{E:0l1} &&\text{(premise)}\\
&\alpha, a, \bar{b}_n\models P'(x,\bar{w}_n)\label{E:0l1-1} &&\text{(premise)}\\
&P' \in \Theta\setminus\{\,R^2,E^2\,\}\label{E:0l1-2} &&\text{(by $\Sigma_\varphi \subseteq \Sigma'$)}\\
&\forall Q \in \Theta\setminus\{\,R^2,E^2\,\}(\alpha, a,
\bar{b}_n\models Q(x, \bar{w}_n)
\Rightarrow \beta, c, \bar{d}_n\models Q(x, \bar{w}_n))\label{E:0l2} &&\text{(from \eqref{E:0l1} by \eqref{E:c1})}\\
&\alpha, a, \bar{b}_n\models P'(x,\bar{w}_n) \Rightarrow \beta,
b,\bar{d}_n \models P'(x,\bar{w}_n)\label{E:0l3} &&\text{(from \eqref{E:0l1-2} and \eqref{E:0l2})}\\
&\beta, c,\bar{d}_n\models P'(x,\bar{w}_n)\label{E:0l4}
&&\text{(from \eqref{E:0l1-1} and \eqref{E:0l3})}
\end{align}

The case $i = \bot$ is obvious.

\emph{Induction step.}

\emph{Case 1.} Let $i(\bar{w}_n) = j(\bar{w}_n) \wedge
k(\bar{w}_n)$. Then $\varphi(x, \bar{w}_n) = ST(j(\bar{w}_n), x)
\wedge ST(k(\bar{w}_n),x)$ and we reason as follows:

\begin{align}
&(\bar{a}_m, a;\bar{b}_n)A(\bar{c}_m, c;\bar{d}_n)\label{E:1l1} &&\text{(premise)}\\
&\alpha, a, \bar{b}_n\models ST(j(\bar{w}_n), x) \wedge
ST(k(\bar{w}_n),x)\label{E:1l2} &&\text{(premise)}\\
&m + n + r(ST(j(\bar{w}_n), x) \wedge ST(k(\bar{w}_n),x)) \leq p + s\label{E:1l3} &&\text{(premise)}\\
&r(ST(j(\bar{w}_n), x)) \leq r(ST(j(\bar{w}_n), x) \wedge ST(k(\bar{w}_n),x))\label{E:1l4} &&\text{(by df of $r$)}\\
&r(ST(k(\bar{w}_n), x)) \leq r(ST(j(\bar{w}_n), x) \wedge ST(k(\bar{w}_n),x))\label{E:1l5} &&\text{(by df of $r$)}\\
&\alpha, a, \bar{b}_n\models ST(j(\bar{w}_n), x)\label{E:1l6} &&\text{(from \eqref{E:1l2})}\\
&\alpha, a, \bar{b}_n\models ST(k(\bar{w}_n), x)\label{E:1l7} &&\text{(from \eqref{E:1l2})}\\
&m + n + r(ST(j(\bar{w}_n), x)) \leq p + s\label{E:1l8} &&\text{(from \eqref{E:1l3} and \eqref{E:1l4})}\\
&m + n + r(ST(k(\bar{w}_n), x)) \leq p + s\label{E:1l9} &&\text{(from \eqref{E:1l3} and \eqref{E:1l5})}\\
&\beta, c,\bar{d}_n \models ST(j(\bar{w}_n), x)\label{E:1l10} &&\text{(from \eqref{E:1l1}, \eqref{E:1l6} and \eqref{E:1l8} by IH)}\\
&\beta, c,\bar{d}_n \models ST(k(\bar{w}_n), x)\label{E:1l11} &&\text{(from \eqref{E:1l1}, \eqref{E:1l7} and \eqref{E:1l9} by IH)}\\
&\beta, c,\bar{d}_n \models ST(j(\bar{w}_n), x) \wedge
ST(k(\bar{w}_n),x)\label{E:1l12} &&\text{(from \eqref{E:1l10} and
\eqref{E:1l11})}
\end{align}

\emph{Case 2.} Let $i(\bar{w}_n) = j(\bar{w}_n) \vee
k(\bar{w}_n)$. Then $\varphi(x, \bar{w}_n) = ST(j(\bar{w}_n), x)
\vee ST(k(\bar{w}_n),x)$ and we have then $\alpha, a,
\bar{b}_n\models ST(j(\bar{w}_n), x) \vee ST(k(\bar{w}_n),x)$.
Assume, without a loss of generality, that $\alpha, a,
\bar{b}_n\models ST(j(\bar{w}_n), x)$. Then we reason as follows:

\begin{align}
&\alpha, a,\bar{b}_n\models ST(j(\bar{w}_n), x)\label{E:2l0} &&\text{(premise)}\\
&(\bar{a}_m, a;\bar{b}_n)A(\bar{c}_m, c;\bar{d}_n)\label{E:2l1} &&\text{(premise)}\\
&m + n + r(ST(j(\bar{w}_n), x)
\vee ST(k(\bar{w}_n),x)) \leq p + s\label{E:2l2} &&\text{(premise)}\\
&r(ST(j(\bar{w}_n), x)) \leq r(ST(j(\bar{w}_n), x)
\vee ST(k(\bar{w}_n),x))\label{E:2l3} &&\text{(by df of $r$)}\\
&m + n + r(ST(j(\bar{w}_n), x)) \leq p + s\label{E:2l4} &&\text{(from \eqref{E:2l2} and \eqref{E:2l3})}\\
&\beta, c,\bar{d}_n \models ST(j(\bar{w}_n), x)\label{E:2l5} &&\text{(from \eqref{E:2l0}, \eqref{E:2l1} and \eqref{E:2l4} by IH)}\\
&\beta, c,\bar{d}_n \models ST(j(\bar{w}_n), x) \vee
ST(k(\bar{w}_n),x)\label{E:2l6} &&\text{(from \eqref{E:2l5})}
\end{align}

\emph{Case 3.} Let $i(\bar{w}_n) = j(\bar{w}_n) \to k(\bar{w}_n)$.
Then
\[
\varphi(x, \bar{w}_n) = \forall y(R(x, y) \to (ST(j(\bar{w}_n), y)
\to ST(k(\bar{w}_n), y))).
\]
Let
\[
\alpha, a,\bar{b}_n \models \forall y(R(x, y) \to
(ST(j(\bar{w}_n), y) \to ST(k(\bar{w}_n), y))),
\]
and let
\[
\beta, c,\bar{d}_n \models \exists y(R(x, y) \wedge
(ST(j(\bar{w}_n), y) \wedge \neg ST(k(\bar{w}_n), y))).
\]
This means that we can choose a $c' \in D(\beta)$ such that
$cR^\beta c'$ and $\beta, c', \bar{d}_n\models ST(j(\bar{w}_n), y)
\wedge \neg ST(k(\bar{w}_n), y)$. We now reason as follows:
\begin{align}
&\beta, c', \bar{d}_n\models ST(j(\bar{w}_n), y) \wedge \neg
ST(k(\bar{w}_n),
y)\label{E:4l00} &&\text{(by choice of $c'$)}\\
&c' \in D(\beta) \wedge cR^\beta c'\label{E:4l0} &&\text{(by choice of $c'$)}\\
&(\bar{a}_m, a;\bar{b}_n)A(\bar{c}_m, c;\bar{d}_n)\label{E:4l1} &&\text{(premise)}\\
&m + n + r(\varphi(x, \bar{w}_n)) \leq p + s\label{E:4l2} &&\text{(premise)}\\
&r(\varphi(x, \bar{w}_n))) \geq 1 \label{E:4l3} &&\text{(by df of $r$)}\\
&m + n < p + s\label{E:4l4} &&\text{(from \eqref{E:4l2} and
\eqref{E:4l3})}\\
&\exists a' \in D(\alpha)(aR^\alpha a' \wedge (\bar{c}_m,
c,c';\bar{d}_n)\hat{A}(\bar{a}_m, a,a';\bar{b}_n))\label{E:4l5}
&&\text{(from \eqref{E:4l0}, \eqref{E:4l1} and \eqref{E:4l4}
by\eqref{E:c2})}
\end{align}
Now choose an $a'$ for which \eqref{E:4l5} is satisfied; we add
the premises following from our choice of $a'$ and continue our
reasoning as follows:
\begin{align}
&a' \in D(\alpha) \wedge aR^\alpha a'\label{E:4l6} &&\text{(by choice of $a'$)}\\
&(\bar{c}_m,
c,c';\bar{d}_n)A(\bar{a}_m, a,a';\bar{b}_n)\label{E:4l7} &&\text{(by choice of $a'$)}\\
&(\bar{a}_m, a,a';\bar{b}_n)A(\bar{c}_m,
c,c';\bar{d}_n)\label{E:4l7-1} &&\text{(by choice of $a'$)}\\
&r(ST(j(\bar{w}_n), y)) \leq r(\varphi(x,\bar{w}_n)) - 1\label{E:4l8} &&\text{(by df of $r$)}\\
&r(ST(k(\bar{w}_n), y)) \leq r(\varphi(x,\bar{w}_n)) - 1\label{E:4l8-1} &&\text{(by df of $r$)}\\
&m + 1 + n + r(ST(j(\bar{w}_n), y)) \leq p + s\label{E:4l9}
&&\text{(from
\eqref{E:4l2} and \eqref{E:4l8})}\\
&m + 1 + n + r(ST(k(\bar{w}_n), y)) \leq p + s\label{E:4l9-1}
&&\text{(from
\eqref{E:4l2} and \eqref{E:4l8-1})}\\
&\alpha, a',\bar{b}_n \models ST(j(\bar{w}_n), x)\label{E:4l10} &&\text{(from \eqref{E:4l00}, \eqref{E:4l7}, \eqref{E:4l9} by IH)}\\
&\alpha, a',\bar{b}_n \models \neg ST(k(\bar{w}_n), x)\label{E:4l10-1} &&\text{(from \eqref{E:4l00}, \eqref{E:4l7-1}, \eqref{E:4l9-1} by IH)}\\
&\alpha, a',\bar{b}_n \models ST(j(\bar{w}_n), y) \wedge \neg ST(k(\bar{w}_n), y)\label{E:4l10-2} &&\text{(from \eqref{E:4l10}, \eqref{E:4l10-1})}\\
&\alpha, a,\bar{b}_n  \models \exists y(R(x, y) \wedge
(ST(j(\bar{w}_n), y) \wedge \neg ST(k(\bar{w}_n),
y)))\label{E:4l11} &&\text{(from \eqref{E:4l6} and
\eqref{E:4l10-2})}
\end{align}
The last line contradicts our initial assumption that
\[
\alpha, a,\bar{b}_n \models \forall y(R(x, y) \to
(ST(j(\bar{w}_n), y) \to ST(k(\bar{w}_n), y))),
\]

\emph{Case 4.} Let $i(\bar{w}_n) = \exists w' j(\bar{w}_n, w')$.
Then
\[
\varphi(x, \bar{w}_n) = \exists w'(E(x, w') \wedge ST(j(\bar{w}_n,
w'), x)).
\]
Let $\alpha, a, \bar{b}_n \models \exists w'(E(x, w') \wedge
ST(j(\bar{w}_n, w'), x))$. This means that we can choose a $b' \in
D(\alpha)$ such that $aE^\alpha b'$ and $\alpha, a, \bar{b}_n,b'
\models ST(j(\bar{w}_n, w'), x)$. We now reason as follows:
\begin{align}
&\alpha, a, \bar{b}_n,b' \models ST(j(\bar{w}_n, w'), x)\label{E:5l0} &&\text{(by choice of $b'$)}\\
&b' \in D(\alpha) \wedge E^\alpha(a, b')\label{E:5l1} &&\text{(by choice of $b'$)}\\
&(\bar{a}_m, a;\bar{b}_n)A(\bar{c}_m, c;\bar{d}_n)\label{E:5l2} &&\text{(premise)}\\
&m + n + r(\varphi(x, \bar{w}_n)) \leq p + s\label{E:5l3} &&\text{(premise)}\\
&r(\varphi(x, \bar{w}_n)) \geq 1 \label{E:5l4} &&\text{(by df of $r$)}\\
&m + n < p + s\label{E:5l5} &&\text{(from \eqref{E:5l3} and
\eqref{E:5l4})}\\
&\exists d' \in D(\beta)(E^\beta(c,d') \wedge (\bar{a}_m, a;
\bar{b}_n, b')A(\bar{c}_m,c;\bar{d}_n, d'))\label{E:5l6}
&&\text{(from \eqref{E:5l1}, \eqref{E:5l2} and \eqref{E:5l5} by
\eqref{E:c3})}
\end{align}
Now choose a $d'$ for which \eqref{E:5l6} is satisfied; we add the
premises following from our choice of $d'$ and continue our
reasoning as follows:
\begin{align}
&d' \in D(\beta) \wedge E^\beta(c,d')\label{E:5l7} &&\text{(by choice of $d'$)}\\
&(\bar{a}_m, a; \bar{b}_n,
b')A(\bar{c}_m,c;\bar{d}_n, d')\label{E:5l8} &&\text{(by choice of $d'$)}\\
&r(ST(j(\bar{w}_n,w'), x)) = r(\varphi(x, \bar{w}_n)) - 1\label{E:5l9} &&\text{(by df of $r$)}\\
&m + n + 1 + r(ST(j(\bar{w}_n,w'), x)) \leq p + s\label{E:5l10}
&&\text{(from
\eqref{E:5l3} and \eqref{E:5l9})}\\
&\beta, c, \bar{d}_n, d' \models ST(j(\bar{w}_n,w'), x)\label{E:5l11} &&\text{(from \eqref{E:5l0}, \eqref{E:5l8}, \eqref{E:5l10} by IH)}\\
&\beta, c, \bar{d}_n \models \exists w'(E(x, w') \wedge
ST(j(\bar{w}_n, w'), x))\label{E:5l12} &&\text{(from \eqref{E:5l7}
and \eqref{E:5l11})}
\end{align}

\emph{Case 5}. Let $i(\bar{w}_n) = \forall w'j(\bar{w}_n, w')$.
Then
\[
\varphi(x, \bar{w}_n) = \forall yw'((R(x,y) \wedge E(y, w')) \to
ST(j(\bar{w}_n, w'), y)).
\]
Let
\[
\alpha, a, \bar{b}_n \models \forall yw'((R(x,y) \wedge E(y, w'))
\to ST(j(\bar{w}_n, w'), y)),
\]
and let
\[
\beta, c, \bar{d}_n \models \exists yw'((R(x,y) \wedge E(y, w'))
\wedge \neg ST(j(\bar{w}_n, w'), y)).
\]
The latter fact means that we can choose some $c',d' \in D(\beta)$
such that $cR^\beta c'$, $E^\beta(c',d')$, and $\beta, c',
\bar{d}_n, d' \models \neg ST(j(\bar{w}_n, w'), y)$. We now reason
as follows:
\begin{align}
&\beta, c',
\bar{d}_n, d' \models \neg ST(j(\bar{w}_n, w'), y)\label{E:6l0} &&\text{(by choice of $c', d'$)}\\
&c' \in D(\beta) \wedge cR^\beta c'\label{E:6l1} &&\text{(by choice of $c'$)}\\
&d' \in D(\beta) \wedge E^\beta(c',d')\label{E:6l1-1} &&\text{(by choice of  $c', d'$)}\\
&(\bar{a}_m, a;\bar{b}_n)A(\bar{c}_m, c;\bar{d}_n)\label{E:6l2} &&\text{(premise)}\\
&m + n + r(\varphi(x, \bar{w}_n)) \leq p + s\label{E:6l3} &&\text{(premise)}\\
&r(\varphi(x, \bar{w}_n))) \geq 2 \label{E:6l4} &&\text{(by df of $r$)}\\
&m + n + 1 < p + s\label{E:6l5} &&\text{(from \eqref{E:6l3} and
\eqref{E:6l4})}
\end{align}
\begin{align}
\exists a'b' \in D(aR^\alpha a'\wedge E^\alpha(a',b') \wedge
(\bar{a}_m, a, a'; \bar{b}_n, b')A(\bar{c}_m,c,
c';\bar{d}_n, d'))\label{E:6l6}\\
\text{(from \eqref{E:6l1}, \eqref{E:6l1-1}, \eqref{E:6l2} and
\eqref{E:6l5} by\eqref{E:c4})}\notag
\end{align}

Then choose $a',b' \in D(\alpha)$ for which \eqref{E:6l6} is
satisfied. We add the premises following from our choice of $a',
b'$ and continue our reasoning as follows:

\begin{align}
&a' \in D(\alpha) \wedge aR^\alpha a'\label{E:6l7} &&\text{(by choice of $a'$)}\\
&b' \in D(\alpha) \wedge E^\alpha(a',b')\label{E:6l7-1} &&\text{(by choice of  $a', b'$)}\\
&(\bar{a}_m, a, a'; \bar{b}_n, b')A(\bar{c}_m,c,
c';\bar{d}_n, d')\label{E:6l8} &&\text{(by choice of $a', b'$)}\\
&r(\neg
ST(j(\bar{w}_n, w'), y)) = r(\varphi(x, \bar{w}_n)) - 2\label{E:6l9} &&\text{(by df of $r$)}\\
&m + 1 + n + 1 + r(\neg ST(j(\bar{w}_n, w'), y)) \leq p +
s\label{E:6l10} &&\text{(from
\eqref{E:6l3} and \eqref{E:6l9})}\\
&\alpha, a', \bar{b}_n,b' \models \neg
ST(j(\bar{w}_n, w'), y)\label{E:6l11} &&\text{(from \eqref{E:6l0}, \eqref{E:6l8}, \eqref{E:6l10} by IH)}\\
&\alpha, a, \bar{b}_n \models \exists yw'((R(x,y) \wedge E(y, w'))
\wedge \neg ST(j(\bar{w}_n, w'), y))\label{E:6l12} &&\text{(from
\eqref{E:6l7}, \eqref{E:6l7-1} and \eqref{E:6l11})}
\end{align}
The last line contradicts our initial assumption that
\[
\alpha, a, \bar{b}_n \models \forall yw'((R(x,y) \wedge E(y, w'))
\to ST(j(\bar{w}_n, w'), y)).
\]
\end{proof}

\begin{definition}\label{D:k-inv}
A formula $\varphi(x, \bar{w}_n)$ is invariant with respect to
$k$-asimulations iff for any $\Theta$ such that $\Sigma_\varphi
\subseteq \Theta$, any two $n$-ary evaluation $\Theta$-points $(M,
a, \bar{b}_n)$ and $(N, c, \bar{d}_n)$, if there exists a $\langle
(M, a, \bar{b}_n), (N, c, \bar{d}_n)\rangle_k$-asimulation $A$ and
$M, a, \bar{b}_n \models \varphi(x, \bar{w}_n)$, then $N,  c,
\bar{d}_n \models \varphi(x, \bar{w}_n)$.
\end{definition}

\begin{corollary}\label{L:c-k-inv}
If $\varphi(x, \bar{w}_n)$ is a standard $x$-translation of an
intuitionistic formula and $r(\varphi) = k$, then $\varphi(x,
\bar{w}_n)$ is invariant with respect to $k$-asimulations.
\end{corollary}

Corollary \ref{L:c-k-inv} immediately follows from Lemma
\ref{L:asim} setting $\alpha = M$, $\beta = N$, $m = 0$, $p = k$,
$t = a$, $\bar{u}_s = \bar{b}_n$, $t' = c$, $\bar{u'}_s =
\bar{d}_n$.

Before we state and prove our main result, we need to mention a
fact from classical model theory of first-order logic.

\begin{lemma}\label{L:fin}
For any finite $\Theta$ and any natural $n, k$ there are, up to
logical equivalence, only finitely many $(\Theta, (x, \bar{w}_n),
k)$-formulas.
\end{lemma}

This fact is proved as Lemma 3.4 in \cite[pp. 189--190]{EFT}.

\begin{definition}\label{D:conj}
Let $\varphi(x, \bar{w}_n)$ be a formula. A conjunction of
$(\Sigma_\varphi,(x, \bar{w}_n), k)$-formulas $\Psi(x, \bar{w}_n)$
is called a complete $(\varphi, (x, \bar{w}_n), k)$-conjunction
iff (1) every conjunct in $\Psi(x)$ is a standard $x$-translation
of an intuitionistic formula; and (2) there is an $n$-ary
evaluation point $(M, a, \bar{b}_n)$ such that $M, a, \bar{b}_n
\models \Psi(x, \bar{w}_n) \wedge \varphi(x, \bar{w}_n)$ and for
any $(\Sigma_\varphi, (x, \bar{w}_n),k)$-formula $\psi(x,
\bar{w}_n)$, if $\psi(x, \bar{w}_n)$ is a standard $x$-translation
of an intuitionistic formula and $M, a, \bar{b}_n \models \psi(x,
\bar{w}_n)$, then $\Psi(x, \bar{w}_n) \models\psi(x, \bar{w}_n)$.
\end{definition}

\begin{lemma}\label{L:conj-ex}
For any formula $\varphi(x, \bar{w}_n)$, any natural $k$, and any
$n$-ary evaluation point $(M, a, \bar{b}_n)$ such that $M, a,
\bar{b}_n \models \varphi(x, \bar{w}_n)$ there is a complete
$(\varphi,(x, \bar{w}_n),k)$-conjunction $\Psi(x, \bar{w}_n)$ such
that $M, a, \bar{b}_n \models \Psi(x, \bar{w}_n) \wedge \varphi(x,
\bar{w}_n)$.
\end{lemma}
\begin{proof} Let $\{\,\psi_1(x, \bar{w}_n)\ldots, \psi_n(x,
\bar{w}_n),\ldots\,\}$ be the set of all $(\Sigma_\varphi, (x,
\bar{w}_n), k)$-formulas that are standard $x$-translations of
intuitionistic formulas true at $(M, a, \bar{b}_n)$. This set is
non-empty since $ST(\bot \to \bot, x)$ will be true at $(M, a,
\bar{b}_n)$. Due to Lemma \ref{L:fin}, we can choose in this set a
non-empty finite subset $\{\,\psi_{i_1}(x, \bar{w}_n)\ldots,
\psi_{i_n}(x, \bar{w}_n)\,\}$ such that any formula from the
bigger set is logically equivalent to (and hence follows from) a
formula in this subset. Therefore, every formula in the bigger set
follows from $\psi_{i_1}(x, \bar{w}_n) \wedge\ldots,
\wedge\psi_{i_n}(x, \bar{w}_n)$ and we also have $M, a, \bar{b}_n
\models \psi_{i_1}(x, \bar{w}_n) \wedge\ldots, \wedge\psi_{i_n}(x,
\bar{w}_n)$, therefore, $\psi_{i_1}(x, \bar{w}_n) \wedge\ldots,
\wedge\psi_{i_n}(x, \bar{w}_n)$ is a complete $(\varphi, (x,
\bar{w}_n), k)$-conjunction. \end{proof}

\begin{lemma}\label{L:conj-fin}
For any formula $\varphi(x, \bar{w}_n)$ and any natural $k$ there
are, up to logical equivalence, only finitely many complete
$(\varphi,(x, \bar{w}_n), k)$-conjunctions.
\end{lemma}
\begin{proof} It suffices to observe that for any formula
$\varphi(x, \bar{w}_n)$ and any natural $k$, a complete
$(\varphi,(x, \bar{w}_n), k)$-conjunction is a $(\Sigma_\varphi,
(x, \bar{w}_n), k)$-formula. Our lemma then follows from Lemma
\ref{L:fin}. \end{proof}

In what follows we adopt the following notation for the fact that
for any sequence $(x, \bar{w}_n)$ of variables all
$(\Sigma_\varphi, (x, \bar{w}_n), k)$-formulas that are standard
translations of intuitionistic formulas true at $(M, a,
\bar{b}_n)$, are also true at $(N, c, \bar{d}_n)$:
\[
(M, a, \bar{b}_n) \leq_{\varphi, n, k} (N, c, \bar{d}_n).
\]

\begin{theorem}\label{L:t1}
Let $r(\varphi(x, \bar{w}_n)) = k$ and let $\varphi(x, \bar{w}_n)$
be invariant with respect to $k$-asimulations. Then $\varphi(x,
\bar{w}_n)$ is equivalent to a standard $x$-translation of an
intuitionistic formula.
\end{theorem}
\begin{proof} We may assume that both $\varphi(x)$ and
$\neg\varphi(x)$ are satisfiable, since both $\bot$ and $\top$ are
obviously invariant with respect to $k$-asimulations and we have,
for example, the following valid formulas:

\begin{align*}
&\bot \leftrightarrow ST(\bot,x), \top \leftrightarrow ST(\bot \to
\bot, x).
\end{align*}

We  may also assume that there are two complete $(\varphi, (x,
\bar{w}_n), k + 2)$-conjunctions $\Psi(x, \bar{w}_n), \Psi'(x,
\bar{w}_n)$ such that $\Psi'(x, \bar{w}_n)\models\Psi(x,
\bar{w}_n)$, and both formulas $\Psi(x, \bar{w}_n) \wedge
\varphi(x, \bar{w}_n)$ and $\Psi'(x, \bar{w}_n) \wedge
\neg\varphi(x, \bar{w}_n)$ are satisfiable.

For suppose otherwise. Then take the set of all complete
$(\varphi, (x, \bar{w}_n),k + 2)$-conjunctions $\Psi(x,
\bar{w}_n)$ such that the formula $\Psi(x, \bar{w}_n) \wedge
\varphi(x, \bar{w}_n)$ is satisfiable. This set is non-empty,
because $\varphi(x, \bar{w}_n)$ is satisfiable, and by Lemma
\ref{L:conj-ex}, it can be satisfied only together with some
complete $(\varphi, (x, \bar{w}_n),k + 2)$-conjunction. Now, using
Lemma \ref{L:conj-fin}, choose in it a finite non-empty subset
$\{\,\Psi_{i_1}(x, \bar{w}_n)\ldots, \Psi_{i_n}(x, \bar{w}_n)\,\}$
such that any complete $(\varphi, (x, \bar{w}_n),k +
2)$-conjunction is equivalent to an element of this subset. We can
show that $\varphi(x, \bar{w}_n)$ is logically equivalent to
$\Psi_{i_1}(x, \bar{w}_n)\vee\ldots \vee\Psi_{i_n}(x, \bar{w}_n)$.
In fact, if $M, a, \bar{b}_n \models \varphi(x, \bar{w}_n)$ then,
by Lemma \ref{L:conj-ex}, at least one complete $(\varphi, (x,
\bar{w}_n),k + 2)$-conjunction is true at $(M, a, \bar{b}_n)$ and
therefore, its equivalent in $\{\,\Psi_{i_1}(x, \bar{w}_n)\ldots,
\Psi_{i_n}(x, \bar{w}_n)\,\}$ is also true at $(M, a, \bar{b}_n)$,
and so, finally we have $M, a, \bar{b}_n \models \Psi_{i_1}(x,
\bar{w}_n)\vee\ldots, \vee\Psi_{i_n}(x, \bar{w}_n)$. In the other
direction, if $M, a, \bar{b}_n \models \Psi_{i_1}(x,
\bar{w}_n)\vee\ldots, \vee\Psi_{i_n}(x, \bar{w}_n)$ then for some
$1 \leq j \leq n$ we have $M, a, \bar{b}_n \models \Psi_{i_j}(x,
\bar{w}_n)$. Then, since $\Psi_{i_j}(x,
\bar{w}_n)\models\Psi_{i_j}(x, \bar{w}_n)$ and since by the choice
of $\Psi_{i_j}(x, \bar{w}_n)$ the formula $\Psi_{i_j}(x,
\bar{w}_n) \wedge \varphi(x, \bar{w}_n)$ is satisfiable, so, by
our assumption, the formula $\Psi_{i_j}(x, \bar{w}_n) \wedge
\neg\varphi(x, \bar{w}_n)$ must be unsatisfiable, and hence
$\varphi(x, \bar{w}_n)$ must follow from $\Psi_{i_j}(x,
\bar{w}_n)$. But in this case we will have $M, a, \bar{b}_n
\models \varphi(x, \bar{w}_n)$ as well. So $\varphi(x, \bar{w}_n)$
is logically equivalent to $\Psi_{i_1}(x, \bar{w}_n)\vee\ldots,
\vee\Psi_{i_n}(x, \bar{w}_n)$ but the latter formula, being a
disjunction of conjunctions of standard $x$-translations of
intuitionistic formulas is itself a standard $x$-translation of an
intuitionistic formula and so we are done.

If, on the other hand, one can take two complete $(\varphi, (x,
\bar{w}_n),k + 2)$-conjunctions $\Psi(x, \bar{w}_n), \Psi'(x,
\bar{w}_n)$ such that $\Psi'(x, \bar{w}_n)\models\Psi(x,
\bar{w}_n)$, and formulas $\Psi(x, \bar{w}_n) \wedge \varphi(x,
\bar{w}_n)$ and $\Psi'(x, \bar{w}_n) \wedge \neg\varphi(x,
\bar{w}_n)$ are satisfiable, we reason as follows.

Take any $n$-ary evaluation $\Sigma_\varphi$-point $(M, a,
\bar{b}_n)$ such that both $M, a, \bar{b}_n \models \Psi(x,
\bar{w}_n) \wedge \varphi(x, \bar{w}_n)$ and for any
$(\Sigma_\varphi, (x, \bar{w}_n),k)$-formula $\psi(x, \bar{w}_n)$,
if $\psi(x, \bar{w}_n)$ is a standard $x$-translation of an
intuitionistic formula and $M, a, \bar{b}_n \models \psi(x,
\bar{w}_n)$, then $\Psi(x, \bar{w}_n) \models\psi(x, \bar{w}_n)$.
Then take any $n$-ary evaluation $\Sigma_\varphi$-point $(N, c,
\bar{d}_n)$ such that $N, c, \bar{d}_n \models \Psi'(x, \bar{w}_n)
\wedge \neg\varphi(x, \bar{w}_n)$.

We can construct a $\langle (M, a, \bar{b}_n), (N, c,
\bar{d}_n)\rangle_k$-asimulation and thus obtain a contradiction
in the following way.

Let $\alpha, \beta \in \{\,M, N\,\}$ and let $(\bar{a'}_m, a',
\bar{b'}_l) \in (D(\alpha)^{m+1} \times D(\alpha)^l)$ and
$(\bar{c'}_m,c'; \bar{d'}_l) \in (D(\beta)^{m+1} \times
D(\beta)^n)$. Then $(\bar{a'}_m, a'; \bar{b'}_l)A(\bar{c'}_m,c';
\bar{d'}_l)$ iff
\[
m + l\leq n + k \wedge (\alpha, a', \bar{b'}_l) \leq_{\varphi, l,
n + k + 2 - m - l} (\beta, c', \bar{d'}_l).
\]

By the choice of $\Psi(x, \bar{w}_n), \Psi'(x, \bar{w}_n)$ and the
independence of truth at an $n$-ary evaluation point from the
choice of free variables in a formula we obviously have $(a;
\bar{b}_n)A(c; \bar{d}_n)$. It remains to verify conditions
\eqref{E:c1}--\eqref{E:c4} of Definition \ref{D:k-asim}.

\emph{Verification of \eqref{E:c1}}. Since the degree of any
atomic formula is $0$, and the above condition implies that $ n +
k + 2 - m - l \geq 2$, it is evident that for any $(\bar{a'}_m,
a'; \bar{b'}_l)A(\bar{c'}_m,c'; \bar{d'}_l)$ and any predicate
letter $P\in\Sigma_\varphi\setminus\{\,R^2,E^2\,\}$ we have
$\alpha, a', \bar{b'}_l \models P(x, \bar{w}_l) \Rightarrow \beta,
c', \bar{d'}_l \models P(x, \bar{w}_l)$.

\emph{Verification of \eqref{E:c2}}. Assume then that for some
$(\bar{a'}_m, a'; \bar{b'}_l)A(\bar{c'}_m,c'; \bar{d'}_l)$ such
that $m + l < n + k$ there exists a $c'' \in D(\beta)$ such that
$c'R^\beta c''$. In this case we will also have $m + 1 + l \leq n
+ k$.

Then consider the following two sets:
\begin{align*}
&\Gamma = \{\,ST(i(\bar{w}_l), x) \mid ST(i(\bar{w}_l), x)\text{
is a $(\Sigma_\varphi, (x, \bar{w}_l), n + k
+ 1 - m - l)$-formula, }\beta, c'', \bar{d'}_l \models ST(i(\bar{w}_l), x)\,\};\\
&\Delta = \{\,ST(i(\bar{w}_l), x) \mid ST(i(\bar{w}_l), x)\text{
is a $(\Sigma_\varphi, (x, \bar{w}_l), n + k + 1 - m -
l)$-formula, }\beta, c'', \bar{d'}_l \models \neg ST(i(\bar{w}_l),
x)\,\}.
\end{align*}
These sets are non-empty, since by our assumption we have $n + k +
1 - m - l\geq 1$. Therefore, as we have $r(ST(\bot, x)) = 0$ and
$r(ST(\bot \to \bot, x)) = 1$, we will also have $ST(\bot, x) \in
\Delta$ and $ST(\bot \to \bot, x) \in \Gamma$. Then, according to
our Lemma \ref{L:fin}, there are finite non-empty sets of logical
equivalents for both $\Gamma$ and $\Delta$. Choosing these finite
sets, we in fact choose some finite
$\{\,ST(i_1(\bar{w}_l),x)\ldots ST(i_t(\bar{w}_l), x)\,\}
\subseteq \Gamma$, $\{\,ST(j_1(\bar{w}_l),x)\ldots
ST(j_u(\bar{w}_l), x)\,\} \subseteq \Delta$ such that
\begin{align*}
&\forall \psi(x,\bar{w}_l) \in
\Gamma(ST(i_1(\bar{w}_l),x)\wedge\ldots \wedge ST(i_t(\bar{w}_l),
x)
\models \psi(x,\bar{w}_l));\\
&\forall \chi(x,\bar{w}_l) \in \Delta(\chi(x,\bar{w}_l)\models
ST(j_1(\bar{w}_l),x)\vee\ldots \vee ST(j_u(\bar{w}_l), x)).
\end{align*}
But then we obtain that the formula
\[
ST((i_1(\bar{w}_l)\wedge\ldots \wedge i_t(\bar{w}_l)) \to
(j_1(\bar{w}_l)\vee\ldots \vee j_u(\bar{w}_l)), x)
\]
is false at $(\beta, c',\bar{d'}_l)$. In fact, $c''$ disproves
this implication for $(\beta, c',\bar{d'}_l)$. But  every formula
both in $\{\,ST(i_1(\bar{w}_l),x)\ldots ST(i_t(\bar{w}_l), x)\,\}$
and $\{\,ST(j_1(\bar{w}_l),x)\ldots ST(j_u(\bar{w}_l), x)\,\}$ is,
by their choice, a $(\Sigma_\varphi, x, n + k + 1 - m -
l)$-formula, and so standard translation of the implication under
consideration must be a $(\Sigma_\varphi,x,n + k + 2 - m -
l)$-formula. Note, further, that by $(\bar{a'}_m, a';
\bar{b'}_l)A(\bar{c'}_m,c'; \bar{d'}_l)$ we must have
\[
(\alpha, a', \bar{b'}_l) \leq_{\varphi, l, n + k + 2 - m - l}
(\beta, c', \bar{d'}_l),
\]
and therefore this implication must be false at $(\alpha,
a',\bar{b'}_l)$ as well. But then take any $a''$ such that
$a'R^\alpha a''$ and $(\alpha, a'', \bar{b'}_l)$ verifies the
conjunction in the antecedent of the formula but falsifies its
consequent. We must conclude then, by the choice of
$\{\,ST(i_1(\bar{w}_l),x)\ldots ST(i_t(\bar{w}_l), x)\,\}$, that
$\alpha, a'',\bar{b'}_l \models \Gamma$ and so, by the definition
of $A$, and given that $m + l + 1 \leq n + k$, that
$(\bar{c'}_m,c', c'';\bar{d'}_l)A(\bar{a'}_m,a', a'';\bar{b'}_l)$.
Since, in addition, $(\alpha, a'', \bar{b'}_l)$ disproves every
formula from $\{\,ST(j_1(\bar{w}_l),x)\ldots ST(j_u(\bar{w}_l),
x)\,\}$ then by the choice of this set we must conclude that every
$(\Sigma_\varphi,x,n + k + 1 - m - l)$-formula that is a standard
$x$-translation of an intuitionistic formula false at $(\beta, c',
\bar{d'}_l)$ is also false at $(\alpha, a'', \bar{b'}_l)$. But
then, again by the definition of $A$, and given the fact that $m +
l + 1 \leq n + k$, we must also have $(\bar{a'}_m,a',
a'';\bar{b'}_l)A(\bar{c'}_m,c', c'';\bar{d'}_l)$, and so condition
\eqref{E:c2} holds.

\emph{Verification of \eqref{E:c3}}. Assume then that for some
$(\bar{a'}_m, a'; \bar{b'}_l)A(\bar{c'}_m,c'; \bar{d'}_l)$ such
that $m + l < n + k$ there exists a $b'' \in D(\alpha)$ such that
$E^\alpha(a', b'')$. In this case we will also have $m + l + 1
\leq n + k$.

Then consider the following set:
\begin{align*}
&\Gamma = \{\,ST(i(\bar{w}_l,w'), x) \mid ST(i, x)\text{ is a
$(\Sigma_\varphi, (x, \bar{w}_l,w'), n + k + 1 - m - l)$-formula,
}\beta, a', \bar{b'}_l,b'' \models ST(i(\bar{w}_l,w'), x)\,\}.
\end{align*}
This set is non-empty, since by our assumption we have $n + k + 1
- m - l\geq 1$. Therefore, as we have $r(ST(\bot \to \bot, x)) =
1$, we will also have $ST(\bot \to \bot, x) \in \Gamma$. Then,
according to our Lemma \ref{L:fin}, there is a finite non-empty
set of logical equivalents for $\Gamma$. Choosing this finite set,
we in fact choose some finite $\{\,ST(i_1(\bar{w}_l,w'),x)\ldots
ST(i_t(\bar{w}_l,w'), x)\,\} \subseteq \Gamma$ such that
\begin{align*}
&\forall \psi(x,\bar{w}_l,w') \in
\Gamma(ST(i_1(\bar{w}_l,w'),x)\wedge\ldots \wedge
ST(i_t(\bar{w}_l,w'), x) \models \psi(x,\bar{w}_l,w')).
\end{align*}
But then we obtain that the formula
\[
ST(\exists w'(i_1(\bar{w}_l,w')\wedge\ldots \wedge
i_t(\bar{w}_l,w')), x)
\]
is true at $(\alpha, a',\bar{b'}_l)$. Moreover, every formula in
$\{\,ST(i_1(\bar{w}_l,w'),x)\ldots ST(i_t(\bar{w}_l,w'), x)\,\}$
is, by their choice, a $(\Sigma_\varphi, x, n + k + 1 - m -
l)$-formula, and so standard translation of the quantified
conjunction under consideration must be a $(\Sigma_\varphi,x,n + k
+ 2 - m - l)$-formula. Since we have, by $(\bar{a'}_m, a';
\bar{b'}_l)A(\bar{c'}_m,c'; \bar{d'}_l)$, that
\[
(\alpha, a', \bar{b'}_l) \leq_{\varphi, l, n + k + 2 - m - l}
(\beta, c', \bar{d'}_l),
\]
then the formula in question must be true at $(\beta,
c',\bar{d'}_l)$ as well. But then take any $d''$ such that
$E^\beta(c',d'')$ and $(\beta, c', \bar{d'}_l,d'')$ verifies a
standard translation of the conjunction after the existential
quantifier. We must conclude then, by the choice of
$\{\,ST(i_1(\bar{w}_l,w'),x)\ldots ST(i_t(\bar{w}_l,w'), x)\,\}$,
that $\beta, c',\bar{d'}_l,d'' \models \Gamma$ and so, by the
definition of $A$, and given that $m + l + 1 \leq n + k$, that
$(\bar{a'}_m,a';\bar{b'}_l,b'')A(\bar{c'}_m,c';\bar{d'}_l,d'')$.

\emph{Verification of \eqref{E:c4}}. Assume then that for some
$(\bar{a'}_m, a', \bar{b'}_l)A(\bar{c'}_m,c', \bar{d'}_l)$ such
that $m + l + 1 < n + k$ there exist some $c'', d'' \in D(\beta)$
such that $c'R^\beta c'' \wedge E^\beta(c'', d'')$, but there are
no $a'',b'' \in D(\alpha)$ such that $a'R^\alpha a'' \wedge
E^\alpha(a'', b'')$ and $(\bar{a'}_m, a',a''; \bar{b'}_l,
b'')A(\bar{c'}_m,c',c'';\bar{d'}_l, d'')$. In this case we will
have $m + 1 + l + 1 \leq n + k$.

Then consider the following set:
\begin{align*}
&\Delta = \{\,ST(i(\bar{w}_l,w'), x) \mid ST(i(\bar{w}_l,w'),
x)\text{ is a $(\Sigma_\varphi, (x, \bar{w}_l,w'), n + k - m -
l)$-formula, }\beta, c'', \bar{d'}_l,d'' \models \neg
ST(i(\bar{w}_l,w'), x)\,\}.
\end{align*}
This set is non-empty, since by our assumption we have $n + k - m
- l\geq 0$. Therefore, as we have $r(ST(\bot, x)) = 0$, we will
also have $ST(\bot, x) \in \Delta$. Then, according to our Lemma
\ref{L:fin}, there is a finite non-empty set of logical
equivalents for $\Delta$. Choosing this finite set, we in fact
choose some finite $\{\,ST(j_1(\bar{w}_l,w'),x)\ldots
ST(j_u(\bar{w}_l,w'), x)\,\} \subseteq \Delta$ such that
\begin{align*}
&\forall \chi(x,\bar{w}_l,w') \in
\Delta(\chi(x,\bar{w}_l,w')\models
ST(j_1(\bar{w}_l,w'),x)\vee\ldots \vee ST(j_u(\bar{w}_l,w'), x)).
\end{align*}
But then we obtain that the formula
\[
ST(\forall w'(j_1(\bar{w}_l,w')\vee\ldots \vee j_u(\bar{w}_l,w')),
x)
\]
is false at $(\beta, c',\bar{d'}_l)$. In fact, $c'',d''$ jointly
disprove standard translation of this universally quantified
disjunction for $(\beta, c',\bar{d'}_l)$. Further, every formula
in $\{\,ST(j_1(\bar{w}_l),x)\ldots ST(j_u(\bar{w}_l), x)\,\}$ is,
by their choice, a $(\Sigma_\varphi, x, n + k - m - l)$-formula,
and so standard translation of the universally quantified
disjunction under consideration must be a $(\Sigma_\varphi,x,n + k
+ 2 - m - l)$-formula. Since we have, by $(\bar{a'}_m, a';
\bar{b'}_l)A(\bar{c'}_m,c'; \bar{d'}_l)$, that
\[
(\alpha, a', \bar{b'}_l) \leq_{\varphi, l, n + k + 2 - m - l}
(\beta, c', \bar{d'}_l),
\]
then the formula in question must be false at $(\alpha,
a',\bar{b'}_l)$ as well. But then take any $a'',b''$ for which we
have $a'R^\alpha a''$ and $E^\alpha(a'',b'')$ such that
$(\alpha,a'', \bar{b'}_l,b'')$ falsifies standard translation of
the disjunction after the quantifier. We must conclude, by the
choice of $\{\,ST(j_1(\bar{w}_l,w'),x)\ldots ST(j_u(\bar{w}_l,w'),
x)\,\}$,
 that every
$(\Sigma_\varphi,x,n + k - m - l)$-formula that is a standard
$x$-translation of an intuitionistic formula false at $(\beta,
c'',\bar{d'}_l,d'')$ is also false at $(\alpha,
a'',\bar{b'}_l,b'')$. But then, again by the definition of $A$,
and given the fact that $m + 1 + l + 1 \leq n + k$, we must also
have $(\bar{a'}_m,a', a'';\bar{b'}_l,b'')A(\bar{c'}_m,c',
c'';\bar{d'}_l,d'')$, so condition \eqref{E:c4} is satisfied.
\end{proof}

\begin{theorem}\label{L:param}
A formula $\varphi(x, \bar{w}_n)$ is equivalent to a standard
$x$-translation of an intuitionistic formula iff there exists $k
\in \mathbb{N}$ such that $\varphi(x, \bar{w}_n)$ is invariant
with respect to $k$-asimulations.
\end{theorem}
\begin{proof} Let $\varphi(x, \bar{w}_n)$ be equivalent to
$ST(i(\bar{w}_n),x)$. Then by Corollary \ref{L:c-k-inv},
$ST(i(\bar{w}_n),x)$ is invariant with respect to
$r(ST(i(\bar{w}_n),x))$-asimulations, and, therefore, so is
$\varphi(x, \bar{w}_n)$. In the other direction, let $\varphi(x,
\bar{w}_n)$ be invariant with respect to $k$-asimulations for some
$k$. If $k \leq r(\varphi(x, \bar{w}_n))$, then every
$r(\varphi(x, \bar{w}_n))$-asimulation is a $k$-asimulation,
therefore, $\varphi(x, \bar{w}_n)$ is invariant with respect to
$r(\varphi(x, \bar{w}_n))$-asimulations, and so, by Theorem
\ref{L:t1}, $\varphi(x, \bar{w}_n)$ is logically equivalent to a
standard $x$-translation of an intuitionistic formula. If, on the
other hand,  $r(\varphi(x, \bar{w}_n)) < k$, then set $l = k -
r(\varphi(x, \bar{w}_n))$ and consider a sequence $\bar{y}_l$ of
variables not occurring in $\varphi(x, \bar{w}_n)$. Formula
$\forall\bar{y}_l\varphi(x, \bar{w}_n)$ is logically equivalent to
$\varphi(x, \bar{w}_n)$, hence $\forall\bar{y}_l\varphi(x,
\bar{w}_n)$ is invariant with respect to $k$-asimulations as well.
But we have $r(\forall\bar{y}_l\varphi(x, \bar{w}_n)) = k$, so, by
Theorem \ref{L:t1}, $\forall\bar{y}_l\varphi(x, \bar{w}_n)$ is
logically equivalent to a standard $x$-translation of an
intuitionistic formula. Hence $\varphi(x, \bar{w}_n)$ is
equivalent to this translation, too. \end{proof}

\section{The main result}\label{S:Main}

We begin by introducing a somewhat simpler, unparametrized version
of asimulation:
\begin{definition}\label{D:asim}
Let $(M,a, \bar{b}_n)$, $(N,c,\bar{d}_n)$ be two $n$-ary
evaluation $\Theta$-points. A binary relation
\[
A \subseteq \bigcup_{n \geq 0}(((D(M)\times D(M)^n) \times (D(N)
\times D(N)^n)) \cup ((D(N)\times D(N)^n) \times (D(M) \times
D(M)^n))),
\]
is called $\langle (M,a, \bar{b}_n),
(N,c,\bar{d}_n)\rangle$-asimulation iff $(a; \bar{b}_n)A(c;
\bar{d}_n)$ and for any $\alpha, \beta \in \{\,M, N\,\}$, any
$(a';\bar{b'}_l) \in D(\alpha) \times D(\alpha)^l$,
$(c';\bar{d'}_l) \in  D(\beta) \times D(\beta)^l$, whenever we
have $(a';\bar{b'}_l)A(c';\bar{d'}_l)$, the following conditions
hold:

\begin{align}
&\forall P \in \Theta\setminus\{\,R^2,E^2\,\}(\alpha, a',
\bar{b'}_l\models P(x, \bar{w}_l)
\Rightarrow \beta, c', \bar{d'}_l\models P(x, \bar{w}_l))\label{E:c11}\\
&(c'' \in D(\beta) \wedge c'R^\beta c'') \Rightarrow\notag\\
&\Rightarrow \exists a'' \in D(\alpha)(a'R^\alpha a'' \wedge
(c'';\bar{d'}_l)\hat{A}(a'';\bar{b'}_l));\label{E:c22}\\
&(b'' \in D(\alpha) \wedge E^\alpha(a', b''))  \Rightarrow\notag\\
&\Rightarrow \exists d'' \in D(\beta)(E^\beta(c',d'') \wedge
(a';\bar{b'}_l, b'')A(c';\bar{d'}_l, d''));\label{E:c33}\\
&(c'', d'' \in D(\beta) \wedge c'R^\beta c''\wedge E^\beta(c'', d'')) \Rightarrow\notag\\
&\Rightarrow \exists a'',b'' \in D(\alpha)(a'R^\alpha a''\wedge
E^\alpha(a'',b'') \wedge (a'';\bar{b'}_l, b'')A(c'';\bar{d'}_l,
d'')).\label{E:c44}
\end{align}
\end{definition}
\begin{lemma}\label{L:k-asim1}
Let $A$ be an $\langle (M,a, \bar{b}_n),
(N,c,\bar{d}_n)\rangle$-asimulation, and let
\[
A' = \{\,\langle(\bar{a'}_m, a';\bar{b'}_l),(\bar{c'}_m,
c';\bar{d'}_l)\rangle\mid (a';\bar{b'}_l)A(c';\bar{d'}_l)\,\}.
\]
Then $A'$ is an $\langle (M,a, \bar{b}_n),
(N,c,\bar{d}_n)\rangle_k$-asimulation for any $k \in \mathbb{N}$.
\end{lemma}
\begin{proof} We obviously have  $(a; \bar{b}_n)A'(c;
\bar{d}_n)$, and since for any $\alpha, \beta \in \{\,M, N\,\}$,
and any $(\bar{a'}_m, a';\bar{b'}_l)$ in $D(\alpha)^{m+1}\times
D(\alpha)^l$, $(\bar{c'}_m, c';\bar{d'}_l)$ in
$D(\beta)^{m+1}\times D(\beta)^l$ such that $(\bar{a'}_m,
a';\bar{b'}_l)A'(\bar{c'}_m, c';\bar{d'}_l)$ we have
$(a';\bar{b'}_l)A(c';\bar{d'}_l)$, condition \eqref{E:c1} for $A'$
follows from the fulfilment of condition \eqref{E:c11} for $A$. So
it remains to verify that the other three conditions hold for $A'$
for every $k$.

\emph{Condition \eqref{E:c2}}: If $(\bar{a'}_m,
a';\bar{b'}_l)A'(\bar{c'}_m, c';\bar{d'}_l)$ then
$(a';\bar{b'}_l)A(c';\bar{d'}_l)$, and if, further, $c'' \in
D(\beta)$ and $c'R^\beta c''$ then by condition \eqref{E:c22} we
can choose $a'' \in D(\alpha)$ such that $a'R^\alpha a''$, and
$(c'';\bar{d'}_l)\hat{A}(a'';\bar{b'}_l)$. But then, by definition
of $A'$ we will also have $(\bar{c'}_m,
c',c'';\bar{d'}_l)\hat{A'}(\bar{a'}_m, a',a'';\bar{b'}_l)$.

\emph{Condition \eqref{E:c3}}: If $(\bar{a'}_m,
a';\bar{b'}_l)A'(\bar{c'}_m, c';\bar{d'}_l)$ then
$(a';\bar{b'}_l)A(c';\bar{d'}_l)$, and if, further, $b'' \in
D(\alpha)$ and $E^\alpha(a',b'')$ then by condition \eqref{E:c33}
we can choose $d'' \in D(\beta)$ such that $E^\beta(c',d'')$, and
$(a';\bar{b'}_l,b'')A(c';\bar{d'}_l,d'')$. But then, by definition
of $A'$ we will also have $(\bar{a'}_m,
a';\bar{b'}_l,b'')A'(\bar{c'}_m, c';\bar{d'}_l,d'')$.

\emph{Condition \eqref{E:c4}}: If $(\bar{a'}_m,
a';\bar{b'}_l)A'(\bar{c'}_m, c';\bar{d'}_l)$ then
$(a';\bar{b'}_l)A(c';\bar{d'}_l)$, and if, further, $c'',d'' \in
D(\beta)$,  $c'R^\beta c''$  and $E^\beta(c'',d'')$ then by
condition \eqref{E:c44} we can choose $a'',b'' \in D(\alpha)$ such
that $a'R^\beta a''$, $E^\alpha(a'',b'')$, and
$(a'';\bar{b'}_l,b'')A(c'';\bar{d'}_l,d'')$. But then, by
definition of $A'$ we will also have $(\bar{a'}_m,
a',a'';\bar{b'}_l,b'')A'(\bar{c'}_m, c',c'';\bar{d'}_l,d'')$.
\end{proof}
\begin{definition}\label{D:inv}
A formula $\varphi(x, \bar{w}_n)$ is invariant with respect to
asimulations iff for any $\Theta$ such that $\Sigma_\varphi
\subseteq \Theta$, any $n$-ary evaluation $\Theta$-points $(M,a,
\bar{b}_n)$ and $(N,c,\bar{d}_n)$, if there exists an $\langle
(M,a, \bar{b}_n), (N,c,\bar{d}_n)\rangle$-asimulation $A$ and $M,
a, \bar{b}_n \models \varphi(x, \bar{w}_n)$, then $N, c, \bar{d}_n
\models \varphi(x, \bar{w}_n)$.
\end{definition}
\begin{corollary}\label{L:c-inv}
If $\varphi(x, \bar{w}_n)$ is equivalent to a standard
$x$-translation of an intuitionistic formula, then $\varphi(x,
\bar{w}_n)$ is invariant with respect to asimulations.
\end{corollary}
\begin{proof} Let $\varphi(x, \bar{w}_n)$ be not invariant with
respect to asimulations, and  let $A$ be an $\langle (M,a,
\bar{b}_n), (N,c,\bar{d}_n)\rangle$-asimulation such that $M, a,
\bar{b}_n \models \varphi(x, \bar{w}_n)$, but not $N, c,\bar{d}_n
\models \varphi(x, \bar{w}_n)$. Let $A'$ be defined as in Lemma
\ref{L:k-asim1}. Then by this Lemma $A'$ is an $\langle (M,a,
\bar{b}_n), (N,c,\bar{d}_n)\rangle_k$-asimulation for any $k \in
\mathbb{N}$. Hence, by Theorem \ref{L:param}, $\varphi(x,
\bar{w}_n)$ cannot be equivalent to a standard $x$-translation of
an intuitionistic formula. \end{proof}

To proceed further, we need to introduce some notions and results
from classical model theory. For a model $M$ and $\bar{a}_n \in
D(M)$ let $[M, \bar{a}_n]$ be the extension of $M$ with
$\bar{a}_n$ as new individual constants denoting themselves. It is
easy to see that there is a simple relation between truth of a
formula at a $\Theta$-evaluation point and truth of its
substitution instance in an extension of the above-mentioned kind;
namely, for any $\Theta$-model $M$, every $\Theta$-formula
$\varphi(\bar{y}_n,\bar{w}_m)$ and any $\bar{a}_n,\bar{b}_m \in
D(M)$ it holds that:

\[
[M, \bar{a}_n], \bar{b}_m \models \varphi(\bar{a}_n,\bar{w}_m)
\Leftrightarrow M, \bar{a}_n, \bar{b}_m \models
\varphi(\bar{y}_n,\bar{w}_m).
\]

We will call a theory of $M$ (and write $Th(M)$) the set of all
first-order sentences true at $M$. We will call an $n$-type of $M$
a set of formulas $\Gamma(\bar{w}_n)$ consistent with $Th(M)$.

\begin{definition}
Let $M$ be a $\Theta$-model. $M$ is $\omega$-saturated iff for all
$k \in \mathbb{N}$ and for all $\bar{a}_n \in D(M)$, every
$k$-type $\Gamma(\bar{w}_k)$ of $[M, \bar{a}_n]$ is satisfiable in
$[M, \bar{a}_n]$.
\end{definition}

Definition of $\omega$-saturation normally requires satisfiability
of $1$-types only. However, our modification is equivalent to the
more familiar version: see e.g. \cite[Lemma 4.31, p. 73]{Doe}.

It is known that every model can be elementarily extended to an
$\omega$-saturated model; in other words, the following lemma
holds:

\begin{lemma}\label{L:ext}
Let $M$ be a $\Theta$-model. Then there is an $\omega$-saturated
extension $N$ of $M$ such that for all $\bar{a}_n \in D(M)$ and
every $\Theta$-formula $\varphi(\bar{w}_n)$:
\[
M, \bar{a}_n \models \varphi(\bar{w}_n) \Leftrightarrow N,
\bar{a}_n \models \varphi(\bar{w}_n).
\]
\end{lemma}
The latter lemma is a trivial corollary of e.g. \cite[Lemma
5.1.14, p. 216]{ChK}.

In what follows we adopt the following notation for the fact that
for any $x$ all $\Theta$-formulas that are standard
$x$-translations of intuitionistic formulas true at $(M,a,
\bar{b}_n)$, are also true at $(N,c, \bar{d}_n)$:
\[
(M,a, \bar{b}_n) \leq_{\Theta} (N,c, \bar{d}_n).
\]

\begin{lemma}\label{L:sat}
Let $\Theta \subseteq \Sigma$, let $M$, $N$ be $\omega$-saturated
$\Theta$-models and let $(M,a, \bar{b}_n) \leq_{\Theta}
(N,c,\bar{d}_n)$. Then relation $A$ such that for any $\alpha,
\beta \in \{\,M, N\,\}$, any $(a';\bar{b'}_l) \in D(\alpha) \times
D(\alpha)^l$, $(c';\bar{d'}_l) \in  D(\beta) \times D(\beta)^l$
\[
(a';\bar{b'}_l)A(c';\bar{d'}_l) \Leftrightarrow (\alpha,a',
\bar{b'}_l) \leq_{\Theta} (\beta,c', \bar{d'}_l)
\]
is an $\langle (M,a,
\bar{b}_n),(N,c,\bar{d}_n)\rangle$-asimulation.\footnote{This
definition of $A$ makes sense only when $D(M) \cap D(N) =
\varnothing$. However, the latter can always be assumed without a
loss of generality.}
\end{lemma}

\begin{proof} Throughout this proof every formula mentioned is supposed to be a $\Theta$-formula.
It is obvious that $(a; \bar{b}_n)A(c;
\bar{d}_n)$, and since for any predicate letter $P$ distinct from
$R^2, E^2$ and variables $x,\bar{w}_n$ formula $P(x,\bar{w}_n)$ is
a standard $x$-translation of an atomic intuitionistic formula,
condition \eqref{E:c11} is trivially satisfied for $A$.

To verify \emph{condition \eqref{E:c22}}, choose any $\alpha,
\beta \in \{\,M, N\,\}$, any $(a';\bar{b'}_l) \in D(\alpha) \times
D(\alpha)^l$, $(c';\bar{d'}_l) \in  D(\beta) \times D(\beta)^l$
such that $(\alpha,
a',\bar{b'}_l)\leq_{\Theta}(\beta,c',\bar{d'}_l)$ and choose any
$c'' \in D(\beta)$ for which we have  $c'R^\beta c''$.

Then choose any variables $x, \bar{w}_n$ and consider the
following two sets:
\begin{align*}
&\Gamma = \{\,i(\bar{w}_l) \mid \beta, c'',\bar{d'}_l \models ST(i(\bar{w}_l), x)\,\};\\
&\Delta = \{\,i(\bar{w}_l) \mid \beta, c'',\bar{d'}_l \models \neg
ST(i(\bar{w}_l), x)\,\}.
\end{align*}
We have by the choice of $\Gamma$, $\Delta$ that for every finite
$\Gamma' \subseteq \Gamma$ and $\Delta' \subseteq \Delta$ the
formula $ST(\bigwedge(\Gamma') \to \bigvee(\Delta'), x)$ is
disproved by $c''$ for $(\beta, c',\bar{d'}_l)$. So, by our
premise that $(\alpha,
a',\bar{b'}_l)\leq_{\Theta}(\beta,c',\bar{d'}_l)$, the standard
translation of every such implication must be false at $(\alpha,
a',\bar{b'}_l)$ as well. This means that every finite subset of
the set
\[
\{\,R(a', x)\,\} \cup \{\,ST(i(\bar{b'}_l),x)\mid
i(\bar{w}_l)\in\Gamma\,\} \cup \{\,\neg ST(i(\bar{b'}_l),x)\mid
i(\bar{w}_l)\in\Delta\,\}
\]
 is satisfiable at $[\alpha, a',\bar{b'}_l]$. (We set $\Delta' = \{\,ST(\bot, x)\,\}$ if the finite set in
question has an empty intersection with $\Delta$ and $\Gamma' =
\{\,ST(\bot \to \bot, x)\,\}$ if it has an empty intersection with
$\Gamma$.) Therefore, by compactness of first-order logic, this
set is consistent with $Th([\alpha, a',\bar{b'}_l])$ and, by
$\omega$-saturation of both $M$ and $N$ it must be satisfied in
$[\alpha, a',\bar{b'}_l]$ by some $a'' \in D(\alpha)$. So for any
such $a''$ we will have $a'R^\alpha a''$ and, moreover
\[
\alpha, a'',\bar{b'}_l \models \{\,ST(i(\bar{w}_l),x)\mid
i(\bar{w}_l)\in\Gamma\,\} \cup \{\,\neg ST(i(\bar{w}_l),x)\mid
i(\bar{w}_l)\in\Delta\,\}.
\]
Thus, by choice of $\Gamma$ and $\Delta$ plus independence of
truth at a pointed model from the choice of free variables in a
formula we will have both $(\alpha,
a'',\bar{b'}_l)\leq_{\Theta}(\beta,c'',\bar{d'}_l)$ and
$(\beta,c'',\bar{d'}_l)\leq_{\Theta}(\alpha, a'',\bar{b'}_l)$ and
condition \eqref{E:c22} is verified.

To verify \emph{condition \eqref{E:c33}}, choose any $\alpha,
\beta \in \{\,M, N\,\}$, any $(a';\bar{b'}_l) \in D(\alpha) \times
D(\alpha)^l$, $(c';\bar{d'}_l) \in  D(\beta) \times D(\beta)^l$
such that $(\alpha,
a',\bar{b'}_l)\leq_{\Theta}(\beta,c',\bar{d'}_l)$ and choose any
$b'' \in D(\alpha)$ for which we have $E^\alpha(a',b'')$.

Then choose any variables $x, \bar{w}_n,w'$ and consider the
following set:
\begin{align*}
&\Gamma = \{\,i(\bar{w}_l,w') \mid\alpha, a',\bar{b'}_l,b''
\models ST(i(\bar{w}_l,w'), x)\,\}.
\end{align*}
We have by the choice of $\Gamma$ that for every finite $\Gamma'
\subseteq \Gamma$ the formula $ST(\exists w'\bigwedge(\Gamma'),
x)$ is verified by $b''$ for $(\alpha, a',\bar{b'}_l)$. So, by our
premise that $(\alpha,
a',\bar{b'}_l)\leq_{\Theta}(\beta,c',\bar{d'}_l)$, the standard
translation of every such quantified conjunction must be true at
$(\beta, c',\bar{d'}_l)$ as well. This means that every finite
subset of the set
\[
\{\,E(c', w')\,\} \cup \{\,ST(i(\bar{d'}_l,w'),c')\mid
i(\bar{w}_l,w')\in\Gamma\,\}
\]
 is satisfiable at $[\beta, c',\bar{d'}_l]$. Therefore, by compactness of first-order logic, this
set is consistent with $Th([\beta, c',\bar{d'}_l])$ and, by
$\omega$-saturation of both $M$ and $N$, it must be satisfied in
$[\beta, c',\bar{d'}_l]$ by some $d'' \in D(\beta)$. So for any
such $d''$ we will have $E^\beta(c',d'')$ and, moreover
\[
\beta, c',\bar{d'}_l,d'' \models \{\,ST(i(\bar{w}_l,w'),x)\mid
i(\bar{w}_l,w')\in\Gamma\,\}.
\]
Thus, by choice of $\Gamma$ plus independence of truth at a
pointed model from the choice of free variables in a formula we
will have $(\alpha,
a',\bar{b'}_l,b'')\leq_{\Theta}(\beta,c',\bar{d'}_l,d'')$
 and condition
\eqref{E:c33} is verified.

To verify \emph{condition \eqref{E:c44}}, choose any $\alpha,
\beta \in \{\,M, N\,\}$, any $(a';\bar{b'}_l) \in D(\alpha) \times
D(\alpha)^l$, $(c';\bar{d'}_l) \in  D(\beta) \times D(\beta)^l$
such that $(\alpha,
a',\bar{b'}_l)\leq_{\Theta}(\beta,c',\bar{d'}_l)$ and choose any
$c'',d'' \in D(\beta)$ for which we have  $c'R^\beta c''$ and
$E^\beta(c'',d'')$.

Then choose any variables $x, \bar{w}_n, w'$ and consider the
following set:
\begin{align*}
&\Delta = \{\,i(\bar{w}_l,w') \mid\beta, c'',\bar{d'}_l,d''
\models \neg ST(i(\bar{w}_l,w'), x)\,\}.
\end{align*}
We have by the choice of $\Delta$ that for every finite $\Delta'
\subseteq \Delta$ the formula $ST(\forall w'\bigvee(\Delta'), x)$
is disproved by $c'',d''$ for $(\beta, c',\bar{d'}_l)$. So, by our
premise that $(a';\bar{b'}_l)\leq_{\Theta}(c';\bar{d'}_l)$, the
standard translation of every such quantified disjunction must be
false at $(\alpha, a',\bar{b'}_l)$ as well. This means that every
finite subset of the set
\[
\{\,R(a', x), E(x,w')\,\} \cup \{\,\neg ST(i(\bar{b'}_l,w'),x)\mid
i(\bar{w}_l,w')\in\Delta\,\}
\]
 is satisfiable at $[\alpha, a',\bar{b'}_l]$. Therefore, by compactness of first-order logic, this
set is consistent with $Th([\alpha, a',\bar{b'}_l])$ and, by
$\omega$-saturation of both $M$ and $N$, it must be satisfied in
$[\alpha, a',\bar{b'}_l]$ by some $a'',b'' \in D(\alpha)$. So for
any such $a''$ and $b''$ we will have $a'R^\alpha a''$,
$E^\alpha(a'',b'')$ and, moreover
\[
\alpha, a'',\bar{b'}_l,b'' \models \{\,\neg
ST(i(\bar{w}_l,w'),x)\mid i(\bar{w}_l,w')\in\Delta\,\}.
\]
Thus, by choice of $\Delta$ plus independence of truth at a
pointed model from the choice of free variables in a formula we
will have $(\alpha,
a'',\bar{b'}_l,b'')\leq_{\Theta}(\beta,c'',\bar{d'}_l,d'')$ and
condition \eqref{E:c44} is verified. \end{proof}

We are prepared now to state and prove our main result.

\begin{theorem}\label{L:main}
Let $\varphi(x, \bar{w}_n)$ be invariant with respect to
asimulations. Then $\varphi(x, \bar{w}_n)$ is equivalent to a
standard $x$-translation of an intuitionistic formula.
\end{theorem}
\begin{proof} We may assume that $\varphi(x, \bar{w}_n)$ is
satisfiable, for $\bot$ is clearly invariant with respect to
asimulations and $\bot \leftrightarrow ST(\bot, x)$ is a valid
formula. In what follows we will write $IC(\varphi(x, \bar{w}_n))$
for the set of $\Sigma_\varphi$-formulas in variables $x,
\bar{w}_n$ that are standard $x$-translations of intuitionistic
formulas following from $\varphi(x, \bar{w}_n))$. For any $n$-ary
evaluation $\Sigma_\varphi$-point $(M, a, \bar{b}_n)$ we will
denote the set of $\Sigma_\varphi$-formulas  in variables $x,
\bar{w}_n$ that are standard $x$-translations of intuitionistic
formulas true at $(M, a, \bar{b}_n)$, or \emph{intuitionistic
$\Sigma_\varphi$-theory} of $(M, a, \bar{b}_n)$ by $IT_\varphi(M,
a, \bar{b}_n)$. It is obvious that for any $n$-ary evaluation
$\Sigma_\varphi$-points $(M, a, \bar{b}_n)$ and $(N, c,
\bar{d}_n)$ we will have $(M, a, \bar{b}_n)
\leq_{\Sigma_\varphi}(N, c, \bar{d}_n)$ if and only if
$IT_\varphi(M, a, \bar{b}_n) \subseteq IT_\varphi(N, c,
\bar{d}_n)$.

Our strategy will be to show that $IC(\varphi(x, \bar{w}_n))
\models \varphi(x, \bar{w}_n)$. Once this is done we will apply
compactness of first-order logic and conclude that $\varphi(x,
\bar{w}_n)$ is equivalent to a finite conjunction of standard
$x$-translations of intuitionistic formulas and hence to a
standard $x$-translation of the corresponding intuitionistic
conjunction.

To show this, take any $n$-ary evaluation $\Sigma_\varphi$-point
$(M, a, \bar{b}_n)$ such that $M, a, \bar{b}_n \models
IC(\varphi(x, \bar{w}_n))$. Such a model exists, because
$\varphi(x, \bar{w}_n)$ is satisfiable and $IC(\varphi(x,
\bar{w}_n))$ will be satisfied in any pointed model satisfying
$\varphi(x, \bar{w}_n)$. Then we can also choose an $n$-ary
evaluation $\Sigma_\varphi$-point $(N, c, \bar{d}_n)$ such that
$N, c, \bar{d}_n \models \varphi(x, \bar{w}_n)$ and $IT_\varphi(N,
c, \bar{d}_n) \subseteq IT_\varphi(M, a, \bar{b}_n)$.

For suppose otherwise. Then for any $n$-ary evaluation
$\Sigma_\varphi$-point $(N, c, \bar{d}_n)$ such that $N, c,
\bar{d}_n \models \varphi(x, \bar{w}_n)$ we can choose an
intuitionistic formula $i_{(N, c, \bar{d}_n)}(\bar{w}_n)$ such
that $ST(i_{(N, c, \bar{d}_n)}(\bar{w}_n), x)$ is a
$\Sigma_\varphi$-formula true at $(N, c, \bar{d}_n)$ but not at
$(M, a, \bar{b}_n)$. Then consider the set
\[
S = \{\,\varphi(x, \bar{w}_n)\,\} \cup \{\,\neg ST(i_{(N, c,
\bar{d}_n)}(\bar{w}_n), x)\mid N, c, \bar{d}_n \models \varphi(x,
\bar{w}_n)\,\}
\]
Let $\{\,\varphi(x, \bar{w}_n), \neg ST(i_1(\bar{w}_n), x)\ldots ,
\neg ST(i_u(\bar{w}_n), x)\,\}$ be a finite subset of this set. If
this set is unsatisfiable, then we must have $\varphi(x) \models
ST(i_1(\bar{w}_n), x)\vee\ldots \vee ST(i_u(\bar{w}_n), x)$, but
then we will also have $(ST(i_1(\bar{w}_n), x)\vee\ldots \vee
ST(i_u(\bar{w}_n), x)) \in IC(\varphi(x, \bar{w}_n)) \subseteq
IT_\varphi(M, a, \bar{b}_n)$, and hence $(ST(i_{(N_1, b_1)},
x)\vee\ldots \vee ST(i_{(N_u, b_u)}, x))$ will be true at $(M, a,
\bar{b}_n)$. But then at least one of $ST(i_1(\bar{w}_n), x)\ldots
,ST(i_u(\bar{w}_n), x)$ must also be true at $(M, a, \bar{b}_n)$,
which contradicts the choice of these formulas. Therefore, every
finite subset of $S$ is satisfiable, and by compactness $S$ itself
is satisfiable as well. But then take any pointed
$\Sigma_\varphi$-model $(N',c', \bar{d'}_n)$ of $S$ and this will
be a model for which we will have both $N', c', \bar{d'}_n \models
ST(i_{(N',c', \bar{d'}_n)}(\bar{w}_n), x)$ by choice of
$i_{(N',c', \bar{d'}_n)}$ and $N',c', \bar{d'}_n \models \neg
ST(i_{(N',c', \bar{d'}_n)}(\bar{w}_n), x)$ by the satisfaction of
$S$, a contradiction.

Therefore, we will assume in the following that $(M, a,
\bar{b}_n)$, $(N,c, \bar{d}_n)$ are $n$-ary evaluation
$\Sigma_\varphi$-points, $M, a, \bar{b}_n \models IC(\varphi(x,
\bar{w}_n))$, $N,c, \bar{d}_n \models \varphi(x, \bar{w}_n)$, and
$IT_\varphi(N,c, \bar{d}_n) \subseteq IT_\varphi(M, a,
\bar{b}_n)$. Then, according to Lemma \ref{L:ext}, consider
$\omega$-saturated elementary extensions $M'$, $N'$ of $M$ and
$N$, respectively. We have:
\begin{align}
&M, a, \bar{b}_n \models \varphi(x, \bar{w}_n) \Leftrightarrow M',
a, \bar{b}_n \models
\varphi(x,\bar{w}_n)\label{E:m1}\\
&N', c, \bar{d}_n \models \varphi(x,\bar{w}_n)\label{E:m2}
\end{align}
Also since $M'$, $N'$ are elementarily equivalent to $M$, $N$ we
have
\[
IT_\varphi(N',c, \bar{d}_n) = IT_\varphi(N,c, \bar{d}_n) \subseteq
IT_\varphi(M, a, \bar{b}_n) = IT_\varphi(M', a, \bar{b}_n).
\]
But then we have $(N',c, \bar{d}_n) \leq_{\Sigma_\varphi} (M',a,
\bar{b}_n)$, and, by $\omega$-saturation of $M'$, $N'$, relation
$A$ as defined in Lemma \ref{L:sat} is an $\langle(N',c,
\bar{d}_n),(M',a, \bar{b}_n)\rangle$-asimulation. But then by
\eqref{E:m2} and asimulation invariance of $\varphi(x,\bar{w}_n)$
we get $M', a, \bar{b}_n \models \varphi(x,\bar{w}_n)$, and
further, by \eqref{E:m1} we conclude that $M, a, \bar{b}_n \models
\varphi(x,\bar{w}_n)$. Therefore, $\varphi(x,\bar{w}_n)$ in fact
follows from $IC(\varphi(x,\bar{w}_n))$. \end{proof}

The following theorem is an immediate consequence of Corollary
\ref{L:c-inv} and Theorem \ref{L:main}:
\begin{theorem}\label{L:final}
A formula $\varphi(x,\bar{w}_n)$ is invariant with respect to
asimulations iff it is equivalent to a standard $x$-translation of
an intuitionistic formula.
\end{theorem}

\section{Criteria for first-order definable classes}\label{S:Rest}

Theorem \ref{L:final} stated above establishes a criterion for the
equivalence of first-order formula to a standard translation of
intuitionistic formula on arbitrary first-order models. But one
may have a special interest in a proper subclass $K$ of the class
of first-order models viewing the models which are not in this
subclass as irrelevant, non-intended etc. In this case one may be
interested in the criterion for equivalence of a given first-order
formula to a standard translation of an intuitionistic predicate
formula \emph{over} this particular subclass. It turns out that if
some parts of this subclass are first-order axiomatizable then
only a slight modification of our general criterion is necessary
to solve this problem.

To tighten up on terminology, we introduce the following
definitions:
\begin{definition}\label{D:k}
Let $K$ be a class of models. Then:
\begin{enumerate}
\item $K(\Theta) = \{\,M \in K\mid K\text{ is a
$\Theta$-model}\,\}$; \item $K(\Theta)$ is first-order
axiomatizable iff there is a set $Ax$ of $\Theta$-sentences, such
that a $\Theta$-model $M$ is in $K$ iff $M \models Ax$; \item A
set $\Gamma$ of $\Theta$-formulas is $K$-satisfiable iff $\Gamma$
is satisfied by some model in $K$; \item A $\Theta$-formula
$\varphi$ $K$-follows from $\Gamma$ $(\Gamma \models_K \varphi)$
iff $\Gamma \cup \{\,\varphi\,\}$ is $K$-unsatisfiable; \item
$\Theta$-formulas $\varphi$ and $\psi$ are $K$-equivalent iff
$\varphi \models_K \psi$ and $\psi \models_K \varphi$.
\end{enumerate}
\end{definition}
It is clear that for any class $K$, such that $Ax$ first-order
axiomatizes $K(\Theta)$, any set $\Gamma$ of $\Theta$-formulas and
any $\Theta$-formula $\varphi$, $\Gamma$ is $K$-satisfiable iff
$\Gamma \cup Ax$ is satisfiable, and $\Gamma \models_K \varphi$
iff $\Gamma \cup Ax \models \varphi$.
\begin{definition}\label{D:int-inv}
A formula $\varphi(x, \bar{w}_n)$ is $K$-invariant with respect to
asimulations iff for any $\Theta$ such that $\Sigma_\varphi
\subseteq \Theta$, any $n$-ary evaluation $\Theta$-points $(M, a,
\bar{b}_n)$ and $(N,c, \bar{d}_n)$, if $M, N \in K$, there exists
an $\langle (M, a, \bar{b}_n),(N,c, \bar{d}_n)\rangle$-asimulation
$A$, and $M, a, \bar{b}_n \models \varphi(x, \bar{w}_n)$, then $N,
c, \bar{d}_n \models \varphi(x, \bar{w}_n)$.
\end{definition}
Now for the criterion of $K$-equivalence:
\begin{theorem}\label{L:int-main}
Let $K$ be a class of first-order models such that $K(\Theta)$ is
first-order axiomatizablefor any finite $\Theta$, and let
$\varphi(x, \bar{w}_n)$ be $K$-invariant with respect to
asimulations. Then $\varphi(x, \bar{w}_n)$ is $K$-equivalent to a
standard $x$-translation of an intuitionistic formula.
\end{theorem}

\begin{proof} Let $Ax_\varphi$ be the set of first-order
sentences that axiomatizes $K(\Sigma_\varphi)$. We may assume that
$\varphi(x, \bar{w}_n)$ is $K(\Sigma_\varphi)$-satisfiable,
otherwise $\varphi(x, \bar{w}_n)$ is $K$-equivalent to $ST(\bot,
x)$ and we are done. In what follows we will write $KC(\varphi(x,
\bar{w}_n))$ for the set of $\Sigma_\varphi$-formulas in variables
$x, \bar{w}_n$ that are standard $x$-translations of
intuitionistic formulas $K$-following from $\varphi(x,\bar{w}_n)$.

Our strategy will be to show that $KC(\varphi(x,\bar{w}_n))
\models_K \varphi(x,\bar{w}_n)$. Once this is done we will
conclude that
\[
Ax_\varphi \cup KC(\varphi(x,\bar{w}_n)) \models
\varphi(x,\bar{w}_n).
\]
Then we apply compactness of first-order logic and conclude that
$\varphi(x,\bar{w}_n)$ is equivalent to a finite conjunction
$\psi_1(x,\bar{w}_n)\wedge\ldots \wedge\psi_m(x,\bar{w}_n)$ of
formulas from this set. But it follows then that
$\varphi(x,\bar{w}_n)$ is $K$-equivalent to the conjunction of the
set $KC(\varphi(x)) \cap \{\,\psi_1(x,\bar{w}_n)\ldots,
\psi_m(x,\bar{w}_n)\,\}$. In fact, by our choice of
$KC(\varphi(x,\bar{w}_n))$ we have
\[
\varphi(x,\bar{w}_n) \models_K \bigwedge(KC(\varphi(x,\bar{w}_n))
\cap \{\,\psi_1(x,\bar{w}_n)\ldots, \psi_m(x,\bar{w}_n)\,\}),
\]
And by our choice of $\psi_1(x,\bar{w}_n)\ldots,
\psi_m(x,,\bar{w}_n)$ we have

\[
Ax_\varphi \cup (KC(\varphi(x,\bar{w}_n)) \cap
\{\,\psi_1(x,\bar{w}_n)\ldots, \psi_m(x,\bar{w}_n)\,\}) \models
\varphi(x,\bar{w}_n)
\]
and hence
\[
KC(\varphi(x,\bar{w}_n)) \cap \{\,\psi_1(x,\bar{w}_n)\ldots,
\psi_m(x,,\bar{w}_n)\,\} \models_K \varphi(x,\bar{w}_n).
\]

To show that $KC(\varphi(x,\bar{w}_n)) \models_K
\varphi(x,\bar{w}_n)$, take any $n$-ary evaluation
$\Sigma_\varphi$-point $(M, a, \bar{b}_n)$ such that $M \in K$ and
$M, a, \bar{b}_n \models KC(\varphi(x,\bar{w}_n))$. Such a model
exists, because $\varphi(x,\bar{w}_n)$ is
$K(\Sigma_\varphi)$-satisfiable and $KC(\varphi(x,\bar{w}_n))$
will be $K$-satisfied in any $n$-ary evaluation
$\Sigma_\varphi$-point satisfying $\varphi(x,\bar{w}_n)$. Then we
can also choose an $n$-ary evaluation $\Sigma_\varphi$-point $(N,
c, \bar{d}_n)$ such that $N \in K$ and $N, c, \bar{d}_n \models
\varphi(x,\bar{w}_n)$ and $IT_\varphi(N, c, \bar{d}_n) \subseteq
IT_\varphi(M, a, \bar{b}_n)$.

For suppose otherwise. Then for any $\Sigma_\varphi$-model $N \in
K$ and any $n$-ary evaluation $\Sigma_\varphi$-point $(N, c,
\bar{d}_n)$ such that $N, c, \bar{d}_n \models
\varphi(x,\bar{w}_n)$ we can choose an intuitionistic formula
$i_{(N, c, \bar{d}_n)}(\bar{w}_n)$ such that $ST(i_{(N, c,
\bar{d}_n)}(\bar{w}_n), x)$ is a $\Sigma_\varphi$-formula true at
$(N, c, \bar{d}_n)$ but not at $(M, a, \bar{b}_n)$. Then consider
the set
\[
S = \{\,\varphi(x,\bar{w}_n)\,\} \cup \{\,\neg ST(i_{(N,c,
\bar{d}_n)}(\bar{w}_n), x)\mid N \in K \wedge N, c, \bar{d}_n
\models \varphi(x,\bar{w}_n)\,\}
\]
Let $\{\,\varphi(x,\bar{w}_n), \neg ST(i_1(\bar{w}_n), x)\ldots ,
\neg ST(i_u(\bar{w}_n), x)\,\}$ be a finite subset of this set. If
this set is $K$-unsatisfiable, then we must have
\[
\varphi(x,\bar{w}_n) \models_K ST(i_1(\bar{w}_n), x)\vee\ldots
\vee ST(i_u(\bar{w}_n), x),
\]
but then we will also have
\[
(ST(i_1(\bar{w}_n), x)\vee\ldots \vee ST(i_u(\bar{w}_n), x)) \in
KC(\varphi(x,\bar{w}_n)) \subseteq IT_\varphi(M, a, \bar{b}_n),
\]
and hence $(ST(i_1(\bar{w}_n), x)\vee\ldots \vee
ST(i_u(\bar{w}_n), x))$ will be true at $(M, a, \bar{b}_n)$. But
then at least one of $ST(i_1(\bar{w}_n), x)\ldots
,ST(i_u(\bar{w}_n), x)$ must also be true at $(M, a, \bar{b}_n)$,
which contradicts the choice of these formulas. Therefore, every
finite subset of $S$ is $K$-satisfiable. But then every finite
subset of the set $S \cup Ax_\varphi$ is satisfiable as well. By
compactness of first-order logic $S \cup Ax_\varphi$ is
satisfiable, hence $S$ is satisfiable over $K$.

But then take any $n$-ary evaluation $\Sigma_\varphi$-point
$(N',c', \bar{d'}_n)$ satisfying $S$ such that $N' \in K$ and this
will be an evaluation point for which we will have both $N', c',
\bar{d'}_n \models ST(i_{(N',c', \bar{d'}_n)}(\bar{w}_n), x)$ by
choice of $i_{(N',c', \bar{d'}_n)}$ and $N',c', \bar{d'}_n \models
\neg ST(i_{(N',c', \bar{d'}_n)}(\bar{w}_n), x)$ by the
satisfaction of $S$, a contradiction.

Therefore, for any given $n$-ary evaluation $\Sigma_\varphi$-point
$(M, a, \bar{b}_n)$ satisfying $ KC(\varphi(x,\bar{w}_n))$ such
that $M \in K$ we can choose an $n$-ary evaluation
$\Sigma_\varphi$-point $(N,c, \bar{d}_n)$ such that $N \in K$, $N,
c, \bar{d}_n\models \varphi(x, \bar{w}_n)$ and $IT_\varphi(N, c,
\bar{d}_n) \subseteq IT_\varphi(M,a, \bar{b}_n)$. Then, reasoning
exactly as in the proof of Theorem \ref{L:main}, we conclude that
$M, a, \bar{b}_n \models \varphi(x, \bar{w}_n)$. Therefore,
$\varphi(x, \bar{w}_n)$ in fact $K$-follows from $KC(\varphi(x,
\bar{w}_n))$. \end{proof}

\begin{theorem}\label{L:int-final}
Let $K$ be a class of first-order models such that for any finite
$\Theta$ the class $K(\Theta)$ is first-order axiomatizable. Then
a formula $\varphi(x, \bar{w}_n)$ is $K$-invariant with respect to
asimulations iff it is $K$-equivalent to a standard
$x$-translation of an intuitionistic formula.
\end{theorem}
\begin{proof} From left to right our theorem follows from
Theorem \ref{L:int-main}. In the other direction, assume that
$\varphi(x, \bar{w}_n)$ is $K$-equivalent to $ST(i(\bar{w}_n),x)$
and assume that for some $\Theta$ such that $\Sigma_\varphi
\subseteq \Theta$, some $n$-ary evaluation $\Theta$-points $(M, a,
\bar{b}_n)$ and $(N, c, \bar{d}_n)$ such that $M,N \in K$, and
some $\langle (M, a, \bar{b}_n),(N, c,
\bar{d}_n)\rangle$-asimulation $A$ we have $M, a, \bar{b}_n
\models \varphi(x, \bar{w}_n)$. Then, by Corollary \ref{L:c-inv}
we have $N, c, \bar{d}_n \models ST(i(\bar{w}_n),x)$, but since
$ST(i(\bar{w}_n),x)$ is $K$-equivalent to $\varphi(x, \bar{w}_n)$
and $N$ is in $K$, we also have $N, c, \bar{d}_n \models\varphi(x,
\bar{w}_n)$. Therefore, $\varphi(x, \bar{w}_n)$ is $K$-invariant
with respect to asimulations. \end{proof}

One obvious instantiation for $K$ would be the class of all
\emph{intuitionistic} models which are normally viewed as intended
models for intuitionistic predicate logic within the framework of
Kripke semantics. A first-order axiomatization for $K(\Theta)$
would be $RT \cup Mon \cup ER \cup Type$, where:
\begin{align*}
&RT = \{\,\forall yR(y,y), \forall yzw((R(y,z) \wedge R(z,w)) \to
R(y,w))\,\};\\
&Mon = \{\,\forall yz\bar{w}_n((P(y, \bar{w}_n) \wedge R(y,z)) \to
P(z, \bar{w}_n))\mid P \in \Theta \setminus \{\,R\,\}\,\};\\
&ER = \{\,\forall x(\exists yE(x,y) \leftrightarrow \neg\exists
yE(y,x)), \forall xy(R(x,y) \to \exists zw(E(x,z) \wedge
E(y,w)))\,\};\\
&Type = \{\,\forall y\bar{z}_n(P(y,\bar{z}_n) \to
\bigwedge^n_{i=1}(E(y,z_i))\mid P \in \Theta \setminus
\{\,R\,\}\,\}.
\end{align*}

Another instantiation for $K$ might be, e.g. the class of
\emph{intuitionistic models with constant domains}. In this case,
if $R^2, E^2 \in \Theta$, a first-order axiomatization for
$K(\Theta)$ is given by $RT \cup Mon \cup ER \cup Type \cup
\{\,CD\,\}$, where
\[
CD = \forall x(\exists yE(y,x) \to \forall yE(y,x)).
\]

Thus our Theorem \ref{L:int-final} yields, among others, a simple
equivalence criterion for these two particular classes of models.

\section{Conclusion and further research}\label{S:final}

Theorems \ref{L:param}, \ref{L:final}, and \ref{L:int-final}
proved above show that the general idea of asimulation for
intuitionistic propositional logic is a faithful analogue of the
idea of world-object bisimulation for modal predicate logic in
many important respects. However, in the predicate case
differences from the corresponding notion of bisimulation are much
more conspicuous than in the propositional case. Thus, if we
introduced `asimulation games' corresponding to the propositional
version of asimulation defined in \cite{Ol} (the main difference
from propositional case being the absence of conditions
\eqref{E:c33} and \eqref{E:c44}) then, given the strength of
condition \eqref{E:c22} we would have these games
indistinguishable from bisimulation games on the segment beginning
from the first move of Duplicator. Every link between worlds
established by this player would have to be symmetrical and the
asymmetry of asimulation would be important only for the intial
pair of worlds.\footnote{This asymmetry would also possibly lead
to exclusion of some successors of the left world of the link from
the domain of the bisimulation game to follow.} This does not hold
in the predicate case. Here, depending on the strategy chosen by
Spoiler, the whole game might be played with the asymmetrical
links between sequences of world and objects; also asymmetry can
be reinstated after the players reach the first symmetrical link
in the game, and the direction of asymmetry can be switched by
moves of the players. All these features show that specific
features of intuitionistic logic can be actualized within the
setting of quantifiers and predicates only, while on the
propositional level one can find but mere rudiments and traces of
them.

One interesting further question lying beyond the scope of the
present paper is the status of the proofs presented above from the
viewpoint of intuitionistic philosophy. It is well-known that
$\omega$-saturated models whose existence is guaranteed by Lemma
\ref{L:ext} might turn out to be uncountable. Hence our proof
might be viewed by a hardcore intuitionist as having no sense at
all. As it happens, there is a way to give another proof of our
main result that looks more favorable to an intuitionistic eye.
This proof uses countable models only and employs the notion of
recursive saturation instead of saturation \emph{simpliciter}.
However, this variant of proof is also a little bit less clear and
more indirect, so we postpone its publication to another occasion.

}
\end{document}